\documentclass[10pt,oneside,reqno]{article}

\usepackage[a4paper,width=15cm,height=24cm]{geometry}

\usepackage{amsmath,amsthm,amssymb}
\usepackage{bbm}
\usepackage{mathrsfs}

\usepackage{natbib}
\usepackage{color}
\usepackage[dvipsnames]{xcolor}
\usepackage{tikz}
\usepackage{tikz-cd}
\usetikzlibrary{shapes.geometric}
\usepackage{subcaption}
\usepackage{caption}
\usepackage[mathlines]{lineno}

\usepackage{filemod}
\usepackage[breaklinks=true]{hyperref}

\theoremstyle{plain}
\newtheorem{theorem}{Theorem}[section]

\newtheorem{lemma}[theorem]{Lemma}
\newtheorem{corollary}[theorem]{Corollary}

\theoremstyle{definition}

\newtheorem{remark}[theorem]{Remark}

\newtheorem{proofclaim}{Claim}
\let\oldproof\proof
\renewcommand{\proof}{%
    \setcounter{proofclaim}{0}%
    \oldproof%
}

\makeatletter

\def\given{\typeout{Command 'given' should only be used within bracket command}}
\newcounter{@bracketlevel}
\def\@bracketfactory#1#2#3#4#5#6{
\expandafter\def\csname#1\endcsname##1{%
\addtocounter{@bracketlevel}{1}%
\global\expandafter\let\csname @middummy\alph{@bracketlevel}\endcsname\given%
\global\def\given{\mskip#5\csname#4\endcsname\vert\mskip#6}\csname#4l\endcsname#2##1\csname#4r\endcsname#3%
\global\expandafter\let\expandafter\given\csname @middummy\alph{@bracketlevel}\endcsname
\addtocounter{@bracketlevel}{-1}}%
}
\def\bracketfactory#1#2#3{%
\@bracketfactory{#1}{#2}{#3}{relax}{1mu plus 0.25mu minus 0.25mu}{0.6mu plus 0.15mu minus 0.15mu}
\@bracketfactory{b#1}{#2}{#3}{big}{1mu plus 0.25mu minus 0.25mu}{0.6mu plus 0.15mu minus 0.15mu}
\@bracketfactory{bb#1}{#2}{#3}{Big}{2.4mu plus 0.8mu minus 0.8mu}{1.8mu plus 0.6mu minus 0.6mu}
\@bracketfactory{bbb#1}{#2}{#3}{bigg}{3.2mu plus 1mu minus 1mu}{2.4mu plus 0.75mu minus 0.75mu}
\@bracketfactory{bbbb#1}{#2}{#3}{Bigg}{4mu plus 1mu minus 1mu}{3mu plus 0.75mu minus 0.75mu}
}
\bracketfactory{clc}{\lbrace}{\rbrace}
\bracketfactory{clr}{(}{)}
\bracketfactory{cls}{[}{]}
\bracketfactory{abs}{\lvert}{\rvert}
\bracketfactory{norm}{\Vert}{\Vert}
\bracketfactory{floor}{\lfloor}{\rfloor}
\bracketfactory{ceil}{\lceil}{\rceil}
\bracketfactory{angle}{\langle}{\rangle}

\newcounter{ctr}\loop\stepcounter{ctr}\edef\X{\@Alph\c@ctr}%
	\expandafter\edef\csname s\X\endcsname{\noexpand\mathscr{\X}}
	\expandafter\edef\csname c\X\endcsname{\noexpand\mathcal{\X}}
	\expandafter\edef\csname b\X\endcsname{\noexpand\boldsymbol{\X}}
	\expandafter\edef\csname I\X\endcsname{\noexpand\mathbbm{\X}}
	\expandafter\edef\csname r\X\endcsname{\noexpand\mathrm{\X}}
\ifnum\thectr<26\repeat

\newcount\minute
\newcount\hour
\newcount\hourMins
\def\now{%
\minute=\time%
\hour=\time \divide \hour by 60%
\hourMins=\hour \multiply\hourMins by 60%
\advance\minute by -\hourMins%
\zeroPadTwo{\the\hour}:\zeroPadTwo{\the\minute}%
}
\def\zeroPadTwo#1{\ifnum #1<10 0\fi#1}

\numberwithin{equation}{section}
\allowdisplaybreaks[4]

\renewcommand\section{\@startsection {section}{1}{\z@}%
{-3.5ex \@plus -1ex \@minus -.2ex}%
{1.3ex \@plus.2ex}%
{\center\small\sc\mathversion{bold}\MakeUppercase}}

\def\subsection#1{\@startsection {subsection}{2}{0pt}%
{-3.5ex \@plus -1ex \@minus -.2ex}%
{1ex \@plus.2ex}%
{\bf\mathversion{bold}}{#1}}

\def\subsubsection#1{\@startsection{subsubsection}{3}{0pt}%
{\medskipamount}%
{-10pt}%
{\normalsize\itshape}{\kern-2.2ex. #1.}}

\def\blfootnote{\xdef\@thefnmark{}\@footnotetext}

\makeatother

\renewcommand{\cite}{\citep}

\def\^#1{\ifmmode {\mathaccent"705E #1} \else {\accent94 #1} \fi}
\def\~#1{\ifmmode {\mathaccent"707E #1} \else {\accent"7E #1} \fi}

\edef\-#1{\noexpand\ifmmode {\noexpand\bar{#1}} \noexpand\else \-#1\noexpand\fi}
\def\>#1{\vec{#1}}
\def\.#1{\dot{#1}}

\def\atop{\@@atop}

\renewcommand{\leq}{\leqslant}
\renewcommand{\geq}{\geqslant}
\renewcommand{\phi}{\varphi}

\usepackage{nomencl}
\makenomenclature

\usepackage{mwe}
\usepackage{enumitem}
\usepackage{xcolor}
\usepackage{mathtools} 

\newcommand{\R}{\mathbb{R}}
\newcommand{\N}{\mathbb{N}}
\newcommand{\E}{\mathbb{E}}

\renewcommand{\P}{\mathbb{P}}

\newcommand{\ind}{\mathbf{1}}

\newcommand{\indep}{\perp\!\!\!\perp}

\newtheorem*{theorem*}{Theorem}
\newtheorem*{lemma*}{Lemma}

\newcommand{\AuthorAffiliation}{University of California, Berkeley, Berkeley, CA 94720, USA}
\newcommand{\CorrespondingAuthorEmail}{\texttt{zachary\_mcnulty@berkeley.edu}}

\title{An Improved Bipartition Cover Bound for the Multispecies Coalescent Model}
\author{%
Zachary McNulty\textsuperscript{a}\\[0.5em]
\small \textsuperscript{a}\AuthorAffiliation\\
\small Email: \CorrespondingAuthorEmail\\
\small Corresponding author: Zachary McNulty%
}
\date{}

\begin{document}

\maketitle

\begin{abstract}
Bipartition cover probabilities quantify whether a collection of gene trees contains every 
bipartition of the underlying species tree, a condition used by some species-tree inference methods to restrict the search space without excluding the true tree. We study this problem under the multispecies coalescent (MSC) model and derive topology-free 
upper bounds on the number of loci needed to obtain a bipartition cover with prescribed confidence, improving
significantly upon existing bounds. Practically, our bounds remain below biologically
realistic numbers of loci across a substantially broader range of parameter settings, expanding their usefulness
for empirical datasets. Theoretically, our analysis sharpens our understanding of coalescence under the MSC model and
develops new asymptotics for these bounds and absorption times under Kingman's coalescent in the natural short branch
regime. Simulations across several species-tree topologies confirm that the new bounds substantially improve upon existing bounds, although their tightness varies with topology.
\end{abstract}


\section{Introduction}
\label{sec:introduction}

Consensus methods have grown in popularity in phylogenetics as 
large genomic datasets have revealed the limits of single-gene 
tree inference. Evolutionary processes such as incomplete lineage
 sorting, gene duplication and loss, and horizontal gene transfer 
 can cause individual gene trees to differ from the true species 
 tree. Consensus and summary-based methods aim to reconcile this
  discordance by aggregating information across many loci, 
  providing statistically consistent estimates of species 
  relationships under realistic models of gene tree variation.

However, even for a relatively small number of species, the space of possible species tree 
topologies is intractably large. Because of this, it is often necessary to reduce the 
search space somehow. One common strategy is to restrict to species trees whose set of
 bipartitions is contained in the set of bipartitions generated by the gene trees, possibly 
 expanding this set with some collection of heuristics. This is the case for many modern 
 phylogenetic algorithms, including FastRFS \cite{VachaspatiWarnow2017FastRFS} and the widely used ASTRAL 
 family: ASTRAL \cite{mirarab2014astral}, ASTRAL-II \cite{mirarab2015astral}, ASTRAL-III \cite{zhang2018astral}, and wASTRAL \cite{Zhang2022_weighted_ASTRAL}.

    In the case the true species tree lies in this restricted space, we say the gene trees form 
    a \textit{bipartition cover} of the species tree. When this occurs, ASTRAL comes with some strong 
    finite-sample guarantees \cite{shekhar2018_how_many_genes_are_enough}. When it does not occur, 
    ASTRAL provides no guarantees on how the estimated tree will relate to the true species tree. As a result, 
    it is important to know how likely such a bipartition cover is to occur when assessing whether these finite-sample guarantees apply.
    Moreover, since an experimenter must decide how many loci to collect before the underlying species 
    tree topology is known, they cannot calibrate the required sample size to that topology. A topology-free upper bound 
    provides a conservative benchmark for the sample size required to obtain a bipartition cover with a prescribed 
    probability. The paper by \citet{uricchio2016_bipartition_cover} develops 
    such a bound as a function only of two key parameters: the number of species $k$ and the minimum
    branch length $T_{\mathrm{min}}$ of the tree.

    In this paper, we identify a few areas where the bound of \citet{uricchio2016_bipartition_cover} 
    is lossy and develop some topology-free improvements. Our main contributions come from careful analysis 
    of various ``worst cases'' of the MSC model: species tree topologies that somehow make coalescence 
    difficult and certain bipartitions challenging to recover. Beyond this, since all these bounds are 
    complicated functions of the underlying parameters, we develop some asymptotics to describe 
    the growth rates of these bounds as functions of these parameters in a few natural regimes.

\subsection{The Multispecies Coalescent Model}
\label{sec:multispecies-coalescent-model}

We start with a brief discussion of the \textit{multispecies coalescent (MSC) model}, which is 
the probabilistic model underlying all the theoretical guarantees of ASTRAL and many other phylogenetic
 inference algorithms. The MSC model describes how gene tree topologies are generated from a 
 given species tree \citep{yang2003bayes,rannala2020msc}. It is in a sense the simplest model 
 that can explain the phenomenon of \textit{incomplete lineage sorting (ILS)}, the observation 
 that the most common gene tree topology can differ from the species tree topology. 
 This phenomenon makes species tree inference difficult: the intuitive ``majority vote'' estimator
  is not consistent. 

The MSC model is essentially just a constrained version of Kingman's coalescent: going backwards in 
time along each branch of the species tree, we assume the surviving gene lineages coalesce according 
to Kingman's coalescent. Namely, each pair of lineages coalesce at rate one. Figure~\ref{fig:msc_model} 
gives an example of the MSC and illustrates how gene tree topologies can differ from the species tree. 
More concretely, let

\begin{equation}
    g_{ij}(T) := \sum_{k=j}^i \frac{\exp(-{k \choose 2}T) (2k-1)(-1)^{k-j}j_{(k-1)} i_{[k]}}{j!(k-j)!i_{(k)}}
    \label{eqn:def_of_g_ij}
\end{equation}

\noindent be the probability that $i$ lineages coalesce into exactly $j$ lineages in
 time $T$ under Kingman's coalescent. Here $a_{(k)} = a(a+1) \cdot ... \cdot (a+k-1)$ and $a_{[k]} = a(a-1)\cdot ... \cdot (a-k+1)$ 
 denote the rising and falling factorial respectively, and the time $T$ is measured in coalescent units. The function $g_{ij}(T)$ 
 is the basis of the original bound in \citet{uricchio2016_bipartition_cover} and
  all the new bounds we develop. In the setting of the MSC model, $T$ represents
   the length of a given branch of the species tree, and $g_{ij}$ describes the 
   distribution of the number of lineages that remain uncoalesced at the end 
   of that branch.

\begin{figure}
    \centering
    \includegraphics[width=0.5\linewidth]{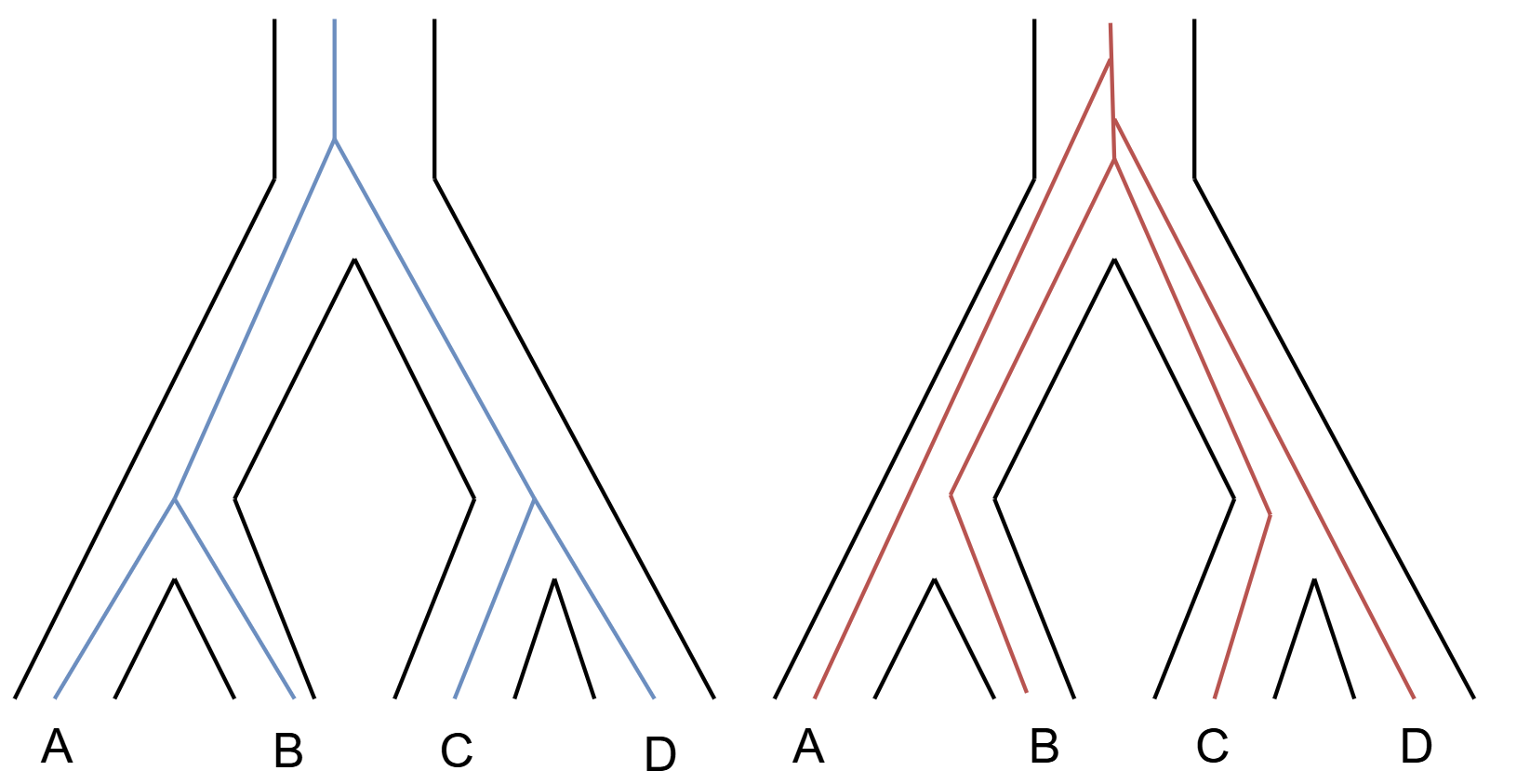}
    \caption{Two different gene trees generated from the same species tree under the multispecies coalescent model. The black lines outline the underlying species tree topology, and the colored lines indicate the gene lineages}
    \label{fig:msc_model}
\end{figure}

\subsection{The Original Bound and Main Results}
\label{sec:original-bound-main-results}

Under the MSC model, \citet{uricchio2016_bipartition_cover} show:

\begin{theorem}[Original Bipartition Cover upper bound] If we have a species tree with $k$ species and minimum branch length $T_{\mathrm{min}}$ (in coalescent units) and a set $\cG$ of at least
\begin{equation}
     M_o(k,T_{\mathrm{min}}) := \frac{\log\left(\frac{1-q}{k-3}\right)}{\log \left(1-g_{k-2, 1}(T_{\mathrm{min}})\right)}
\end{equation}
\noindent gene trees sampled independently under the MSC model, then $\cG$ forms a bipartition cover with probability at least $q$.
\label{thm:uricchios_bipartition_cover_bound}
\end{theorem}

\noindent In this work, we aim to develop a topology-free improvement of this upper bound. 
Our work relies on careful analysis of various ``worst-case'' topologies when it comes to phylogenetic inference. For this purpose, we define two tree types: a \textit{caterpillar tree} and a \textit{balanced tree}.

A \textit{balanced tree} is any rooted tree so that for every vertex the number of leaves in its two subtrees differs by at most one. A \textit{caterpillar tree} is a tree where every non-leaf vertex has one child who is a leaf. Figure \ref{fig:caterpillar_and_balanced_tree} provides a simple illustration. 

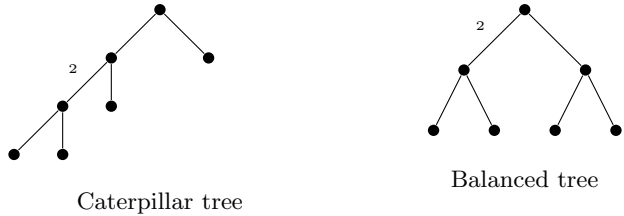
\begin{figure}[h!]
\centering
\begin{tikzpicture}[scale=0.8]

\begin{scope}[xshift=-3cm]
\node[circle, fill=black, inner sep=1.5pt] (root) at (0,0) {};
\node[circle, fill=black, inner sep=1.5pt] (n1) at (-0.8,-0.8) {};
\node[circle, fill=black, inner sep=1.5pt] (n2) at (0.8,-0.8) {};
\node[circle, fill=black, inner sep=1.5pt] (n3) at (-1.6,-1.6) {};
\node[circle, fill=black, inner sep=1.5pt] (n4) at (-0.8,-1.6) {};
\node[circle, fill=black, inner sep=1.5pt] (n5) at (-2.4,-2.4) {};
\node[circle, fill=black, inner sep=1.5pt] (n6) at (-1.6,-2.4) {};

\draw (root) -- (n1);
\draw (root) -- (n2);
\draw (n1) -- (n3) ;
\draw (n1) -- (n4);
\draw (n3) -- (n5);
\draw (n3) -- (n6);

\node at (0,-3.2) {\small Caterpillar tree};
\end{scope}

\begin{scope}[xshift=3cm]
\node[circle, fill=black, inner sep=1.5pt] (root) at (0,0) {};
\node[circle, fill=black, inner sep=1.5pt] (n1) at (-1,-1) {};
\node[circle, fill=black, inner sep=1.5pt] (n2) at (1,-1) {};
\node[circle, fill=black, inner sep=1.5pt] (n3) at (-1.5,-2) {};
\node[circle, fill=black, inner sep=1.5pt] (n4) at (-0.5,-2) {};
\node[circle, fill=black, inner sep=1.5pt] (n5) at (0.5,-2) {};
\node[circle, fill=black, inner sep=1.5pt] (n6) at (1.5,-2) {};

\draw (root) -- (n1) ;
\draw (root) -- (n2);
\draw (n1) -- (n3);
\draw (n1) -- (n4);
\draw (n2) -- (n5);
\draw (n2) -- (n6);

\node at (0,-2.8) {\small Balanced tree};
\end{scope}
\end{tikzpicture}

\caption{The two main extremal tree topologies: caterpillar and balanced trees}

\label{fig:caterpillar_and_balanced_tree}
\end{figure}


These two tree topologies represent opposite extremes for coalescent behavior: one maximally balanced and the other maximally unbalanced. Caterpillar trees present a combinatorial bottleneck for phylogenetic inference, as many of their bipartitions involve very large subsets of taxa. Balanced trees, however, present a more subtle and severe difficulty: they induce a \emph{coalescent bottleneck}, in which lineages are so evenly dispersed that coalescence is systematically delayed.

This distinction is made precise in Lemma~\ref{lem:caterpillar_maximize_increasing_sums}, which identifies caterpillar trees as extremal for increasing functions of descendant counts, and in Lemma~\ref{lem:balanced_is_worse_case}, which shows that balanced trees are stochastically worst-case for the number of surviving lineages under the multispecies coalescent. Together, these lemmas underpin our main result, which appears as Theorem~\ref{thm:improved_bipartition_cover_bound_balanced} below:

\begin{theorem*}[Main result]
For $\ell\geq 2$, let $p_\ell(T)$ denote the probability that the lineages descending 
from a balanced $\ell$-taxon tree, with all branch lengths equal to $T$, coalesce to a single 
lineage. Then $n$ independently sampled gene trees form a bipartition cover with probability at least $q$ whenever
\[
    \sum_{\ell=2}^{k-2}\bigl(1-p_\ell(T_{\min})\bigr)^n\leq 1-q.
\]
In particular, it suffices that
\[
    n\geq
    \frac{\log\!\left(\frac{k-3}{1-q}\right)}
         {-\log\!\left(1-p_{k-2}(T_{\min})\right)}.
\]
\end{theorem*}

\noindent Due to the recursive structure of a balanced tree, these probabilities $p_\ell$ are easy 
to compute. Empirically, we observe this improves over Theorem \ref{thm:uricchios_bipartition_cover_bound} by
several orders of magnitude in some regimes, especially when the minimum branch length $T_{\mathrm{min}}$ is small. Moreover, through careful asymptotic analysis of the coalescent 
probabilities, we show that, in the fixed-$T$, large-$k$ regime, the improvement factor is 
bounded below asymptotically by $\pi^2/(2T)$ as $T \downarrow 0$.
 Lemma \ref{lem:improvement_ratio_asymptotics} below makes this precise. 

\section{Results and Discussion}
\label{sec:results-discussion}

To guide our work, we outline the proof of the original bound, and identify areas where it is lossy.

\begin{proof}[Proof Outline of Theorem \ref{thm:uricchios_bipartition_cover_bound}]
Every tree contains the trivial bipartitions $\{x\} | (S - \{x\})$, so we may restrict attention to nontrivial bipartitions. There are $k-3$ nontrivial bipartitions in a species tree: one for each internal edge\footnote{There's a slight subtlety: there are $k-2$ internal edges, but the two edges adjacent to the root induce the same, possibly trivial, bipartition}.
\begin{enumerate}
    \item By union bound we can just bound the probability a specific species tree bipartition occurs in the gene trees.

    \item  By independence of gene trees we can reduce to the case of a single gene tree. 

    \item Suppose $\phi_e$ is the bipartition associated to edge $e$ in a species tree. One way $\phi_e$ can occur in the gene tree is if all lineages below $e$ coalesce by the time they exit $e$. If there are $k_e$ such lineages, we can lower bound this probability by the probability $k_e$ lineages coalesce down to one along just the single edge $e$. 

    \item The worst-case for the previous step is if edge $e$  is as short as possible (length $T_{\mathrm{min}})$ and the number of lineages that have to coalesce is as large as possible $(k_{e} = k-2$ since $\phi_e$ is nontrivial). 
\end{enumerate}

If we let $E_n$ be the event all bipartitions occur in $n$ gene trees and $E_{i,n}$ be the event bipartition $\phi_i$ occurs in at least one of the $n$ gene trees, then combining the above gives

\begin{align*}
    \P(E_n) & \geq 1 - \sum_{i=1}^{k-3} (1-\P(E_{i,n}))\\
    &  \geq 1 - (k-3)  \cdot \max_i (1-\P(E_{i,n})) \\
    & = 1 - (k-3) (1- \min_i \P(E_{i,1}))^n \\
    & \geq  1 - (k-3) (1- g_{k-2, 1}(T_{\mathrm{min}}))^n
\end{align*}

\noindent from which we can invert to find the bound on $n$.
\end{proof}

\subsection{Bookkeeping of Descendant Counts}
\label{sec:first-improvement-descendant-counts}

One step at which the proof of Theorem \ref{thm:uricchios_bipartition_cover_bound} is lossy is that we essentially bound

\[\P(E_{i,1}) \geq g_{\alpha_{i}, 1}(T_{e_i}) \geq g_{k-2, 1}(T_{\mathrm{min}})\]

\noindent where $\alpha_{i}$ is the size of the part of bipartition $\phi_i$ that is ``below the root''. Namely, if $\phi_i$ is generated by cutting edge $e_i$, then $\alpha_{i}$ is the number of leaves below $e_i$, in the part not containing the root. Since both edges adjacent to the root generate the same bipartition, there is some ambiguity in the definition of our descendant count set $\{\alpha_i\}$. To resolve this, we make the decision to always choose the edge adjacent to the root with the least descendants. An example is given in Figure \ref{fig:example_descendants_alpha_i}.

\begin{figure}[h]
\centering
\begin{tikzpicture}[scale=0.7]

\node[circle, fill=black, inner sep=1.5pt] (root) at (0,0) {};
\node[circle, fill=black, inner sep=1.5pt] (n1) at (-2,-1.2) {};
\node[circle, fill=black, inner sep=1.5pt] (n2) at (2,-1.2) {};
\node[circle, fill=black, inner sep=1.5pt] (n3) at (-3,-2.4) {};
\node[circle, fill=black, inner sep=1.5pt] (n4) at (-1,-2.4) {};
\node[circle, fill=black, inner sep=1.5pt] (n5) at (1,-2.4) {};
\node[circle, fill=black, inner sep=1.5pt] (n6) at (3,-2.4) {};
\node[circle, fill=black, inner sep=1.5pt] (n7) at (-3.5,-3.6) {};
\node[circle, fill=black, inner sep=1.5pt] (n8) at (-2.5,-3.6) {};
\node[circle, fill=black, inner sep=1.5pt] (n9) at (-1.5,-3.6) {};
\node[circle, fill=black, inner sep=1.5pt] (n10) at (-0.5,-3.6) {};
\node[circle, fill=black, inner sep=1.5pt] (n11) at (0.5,-3.6) {};
\node[circle, fill=black, inner sep=1.5pt] (n12) at (1.5,-3.6) {};

\draw (root) -- (n1) node[midway, above left] {\tiny 4};
\draw (root) -- (n2) node[midway, above right] {\tiny 3};
\draw (n1) -- (n3) node[midway, above left] {\tiny 2};
\draw (n1) -- (n4) node[midway, above right] {\tiny 2};
\draw (n2) -- (n5) node[midway, above left] {\tiny 2};
\draw (n2) -- (n6);
\draw (n3) -- (n7);
\draw (n3) -- (n8);
\draw (n4) -- (n9);
\draw (n4) -- (n10);
\draw (n5) -- (n11);
\draw (n5) -- (n12);
\end{tikzpicture}
\caption{Example of the descendant counts $\{\alpha_i\}$ associated to the nontrivial bipartitions/edges of a phylogenetic tree. Note there is some subjectivity in deciding which edge adjacent to the root to take: depending on our choice our descendant counts would either be $\{\alpha_i\} = \{4,2,2,2\}$ or $\{\alpha_i\} = \{3,2,2,2\}$. To remove ambiguity, we always choose the edge with the least descendants.}
\label{fig:example_descendants_alpha_i}
\end{figure}
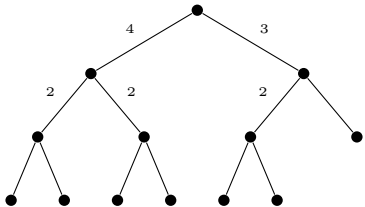

For most edges $e_i$, the number of descendants $\alpha_{i}$ is far smaller than the total number of species $k$. 
By Lemma \ref{lem:cdf_of_Zi}, we know $\ell \mapsto g_{\ell 1}(T_{\mathrm{min}})$ is a 
decreasing function of $\ell$: the more lineages we start with, the harder it is to coalesce down to one.
As a result, this second approximation $g_{\alpha_{i}, 1} \geq g_{k-2, 1}$ can be quite lossy. Nonetheless, to derive 
a topology-free bound we must assume these $\alpha_i$ are in a sense as large
as possible. The following makes this precise.

\begin{lemma}[Caterpillars Maximize Increasing Sums]
Suppose $\{\alpha_{i}\}_{i=1}^{k-3}$ are the descendant counts of our rooted binary tree $T$. Moreover, suppose $f: \N \to \R$ is nondecreasing. Then $\sum_i f(\alpha_{i})$ is maximized when $T$ is the caterpillar tree, in which case $\{\alpha_{i}\} = \{2,3,..., k-2\}$. 

\label{lem:caterpillar_maximize_increasing_sums}
\end{lemma}

\noindent The proof of Lemma \ref{lem:caterpillar_maximize_increasing_sums} can be found 
in Appendix~\ref{app:caterpillar-maximization-proof}. 
This result implies:

\begin{corollary}
Suppose $h: \N \to [0,1]$ is decreasing and $\P(E_{i,1}) \geq h(k_i)$. Then
\[\P(E_n) \geq 1 - \sum_{\ell = 2}^{k-2} \Big[1 - h(\ell)\Big]^n.\]
\label{corr:bipartitions_corollary}    
\end{corollary}

\begin{proof}[Proof of \ref{corr:bipartitions_corollary}]
    By union bound and independence of the gene trees

    \[1-\P(E_n)  \leq  \sum_{i=1}^{k-3} (1-\P(E_{i,n})) =  \sum_{i=1}^{k-3} \Big[1-\P(E_{i,1})\Big]^n \leq  \sum_{i=1}^{k-3} \Big[1-h(k_i)\Big]^n. \]
\noindent  As $h$ is decreasing and bounded in $[0,1]$, the function $f(x) = [1-h(x)]^n$ is increasing. Applying Lemma \ref{lem:caterpillar_maximize_increasing_sums} to this $f$ gives the desired result.
\end{proof}

\noindent From this, the main bound follows shortly.

\begin{theorem}
    Suppose $h: \N \to [0,1]$ is a decreasing function and $\P(E_{i,1}) \geq h(k_i)$. If we have a species tree with $k$ species and a set $\cG$ of at least
\begin{equation}
    n_h := \inf\left\{n \in \N : \sum_{\ell=2}^{k-2}\left[1-h(\ell)\right]^n \leq 1-q\right\}
\end{equation}
\noindent gene trees sampled independently under the MSC model, then $\cG$ forms a bipartition cover of the species tree with probability at least $q$.
\label{thm:caterpillar_bound}
\end{theorem}

\noindent Note that if $h$ is independent of the species tree topology, 
then so is the bound in Theorem~\ref{thm:caterpillar_bound}. If we let 
$h(x) = g_{x,1}(T_{\mathrm{min}}) \in [0,1]$, which is decreasing by Corollary~\ref{corr:g_i1_is_decreasing},
 then Theorem~\ref{thm:caterpillar_bound} and the original bound $\P(E_{i,1}) \geq g_{k_i,1}(T_{\mathrm{min}})$ 
 of \citet{uricchio2016_bipartition_cover} give a new topology-free bound.

\begin{corollary}[Caterpillar Bipartition Cover Bound]
If we have a species tree with $k$ species and a set $\cG$ of at least
\begin{equation}
    M_c(k, T_{\mathrm{min}}):=  \inf\left\{n \in \N : \sum_{\ell=2}^{k-2}\left[1-g_{\ell,1 }(T_{\mathrm{min}})\right]^n \leq 1-q\right\}
    \label{eqn:caterpillar_bound_eqn}
\end{equation}
\noindent gene trees sampled independently under the MSC model, then $\cG$ forms a bipartition cover of the species tree with probability at least $q$. Here $T_{\mathrm{min}}$ is the minimum branch length (in coalescent units) in the species tree.
\label{cor:improved_bipartition_cover_bound_caterpillar}
\end{corollary}

\noindent As with the original bound~\ref{thm:uricchios_bipartition_cover_bound}, Equation~\eqref{eqn:caterpillar_bound_eqn} 
depends on the species tree only through $k$ and $T_{\mathrm{min}}$. Heuristically, this just 
allows us to replace the worst-case $g_{k-2, 1}(T_{\mathrm{min}})$ with more of an 
average case across all of $\{g_{\ell, 1}(T_{\mathrm{min}})\}_{\ell = 2}^{k-2}.$ 

\subsection{Accounting for Coalescent Events Below an Edge}
\label{sec:coalescent-events-below-edge}

\subsubsection{Local balancing} Theorem \ref{thm:caterpillar_bound} suggests that one avenue for improvement is to
obtain a tighter bound $\P(E_{i,1}) \geq h(k_i)$ than the one obtained by choosing $h(x) = g_{x,1}(T_{\mathrm{min}})$.
For small values of $T_{\mathrm{min}}$ the bound in Theorem \ref{cor:improved_bipartition_cover_bound_caterpillar}
is dominated by the terms $g_{k_i,1}$ for which $k_i$ are quite large. These are the bipartitions
 corresponding to edges $e$ close to the root of the tree which have many descendants. In \citet{uricchio2016_bipartition_cover}, 
 the authors essentially assume \textit{none} of the descendants coalesce before they reach edge $e$. This is conservative, especially when $e$ is far from the leaves:
  in this setting, descendants have a long distance to travel through the tree and potentially coalesce.

Hence to improve this bound we should try to account for coalescent events that occur below the edge $e$. Namely, we ought to somehow replace $g_{k_e,1}(T_{\mathrm{min}})$ by $\E g_{X_e, 1}(T_{\mathrm{min}})$ where $X_e$ is the (random) number of remaining lineages entering edge $e$. Clearly $\P(E_{i,1}) \geq \E g_{X_{e_i}, 1}(T_{\mathrm{min}})$ by the same logic that Theorem \ref{thm:uricchios_bipartition_cover_bound} employs. However, $X_e$ will of course depend on the topology of the species tree below the edge $e$. Hence if we hope to develop a topology-free bound, we must somehow develop a notion of the worst-case topology in this setting. 

Recall that random variable $X$ is said to \textit{first-order stochastically dominate} $Y$, and denote this by $X \geq_{st} Y$, 
if $\P(X > t) \geq \P(Y > t)$ for all $t$, or equivalently if $\E u(X) \geq \E u(Y)$ for all nondecreasing 
functions $u$. The random variable $X$ is said to \textit{second-order stochastically dominate} 
$Y$ if $\E u(X) \geq \E u(Y)$ for all nondecreasing concave functions, and we denote this
 by $X \geq_{sst} Y$. First-order stochastic dominance naturally implies second-order.

 Since $x \mapsto g_{x,1}(T)$ is decreasing, convex (Corollary \ref{corr:g_i1_is_decreasing}, Lemma \ref{lem:g_i1_is_convex}), we immediately get
\begin{lemma}
If $U_{e_i} \geq_{sst} X_{e_i}$ then $\P(E_{i,1}) \geq \E g_{U_{e_i}, 1}(T_{\mathrm{min}})$. 
\label{lem:bounding_prob_with_second_order_stochastic_dominance}
\end{lemma}

\begin{remark}
    If $U_e$ is a function of $k_e$ alone and is (second order) stochastically increasing in this parameter,
    applying Theorem~\ref{thm:caterpillar_bound} with $h(\ell) = \E g_{U_\ell, 1}(T_{\mathrm{min}})$ yields
    a new topology-free bound.
\end{remark}

 To apply this observation, we begin by analyzing the coalescent events 
occurring in the two edges directly below $e$. 

Let $Z_i^T$ be the random variable with distribution
 $\P(Z_i^T = j) = g_{ij}(T)$. Namely, if we start $i$ lineages then $Z_i^T$ represents the (random) 
 number of lineages that remain uncoalesced after running Kingman's coalescent 
 for time $T$. When we drop the $T$ superscript, it is implicitly assumed that $T = T_{\mathrm{min}}$. 
 For a positive integer-valued random variable $A$, let $Z_A^T$ denote the independent mixture
    \[\P(Z_A^T = i) = \sum_{\ell \geq 1} \P(A = \ell) \P(Z_\ell^T = i).\]
 \noindent When we write expressions such as $Z_A^T + Z_B^T$ including multiple such mixtures, 
 we implicitly mean underlying families $\{Z_\ell\}_{\ell \geq 1}, \{Z_\ell'\}_{\ell\geq 1}$ are
 independent of each other.

Suppose below the edge $e$ our tree splits into two subtrees of size $m, k-m$. Clearly then
\[X_e \leq_{st}  S := Z_{m}^{T_1} + Z_{k-m}^{T_2},\]
\noindent where $T_1, T_2$ are the lengths of the branches connecting the two subtrees 
to $e$.  Moreover
    \[Z_{m}^{T_1} + Z_{k-m}^{T_2} \leq_{st} Z_m^{T_{\mathrm{min}}} + Z_{k-m}^{T_{\mathrm{min}}},\]
\noindent by Lemma~\ref{lem:cdf_of_Zi}. Figure~\ref{fig:coalescent_tree} illustrates the setup.

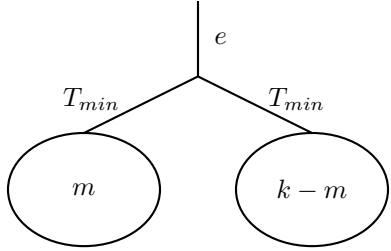
\begin{figure}[t!]
\centering
\begin{tikzpicture}[scale=1]

\coordinate (root) at (0, 3);
\coordinate (left_child) at (-1.5, 1.5);
\coordinate (right_child) at (1.5, 1.5);

\draw[thick] (root) -- (0, 4);
\node at (0.3, 3.5) {$e$};

\draw[thick] (root) -- (-1.5, 2.25);  
\draw[thick] (root) -- (1.5, 2.25);   

\node[draw, ellipse, minimum width=2cm, minimum height=1.5cm, thick] at (left_child) {};
\node at (left_child) {$m$};

\node[draw, ellipse, minimum width=2cm, minimum height=1.5cm, thick] at (right_child) {};
\node at (right_child) {$k-m$};

\node at (-1.4, 2.7) {$T_{\mathrm{min}}$};
\node at (1.3, 2.7) {$T_{\mathrm{min}}$};

\end{tikzpicture}
\caption{Coalescent tree structure showing edge $e$ with two subtrees of sizes $m$ and $k-m$. The subtrees are connected to $e$ via a branch with minimal length $T_{\mathrm{min}}$.}
\label{fig:coalescent_tree}
\end{figure}

\begin{lemma}[Deterministic Balancing Lemma] Let $k$ and $T$ be fixed. Then
    \[\quad Z^T_{i} + Z^T_{k-i} \leq_{st}  Z^T_{j} + Z^T_{k-j}, \quad\forall 1 \leq i \leq j \leq \frac{k}{2}.\]
    \noindent In particular, the stochastic maximizer is $i = \lfloor k/2\rfloor$. 
    \label{lem:deterministic_balancing_lemma}
\end{lemma}

\noindent The proof can be found in Appendix~\ref{app:deterministic-balancing-proof}.

Heuristically, $i = \lfloor k/2 \rfloor$, where the two subtrees in Figure~\ref{fig:coalescent_tree} are as evenly 
balanced as possible, is the maximizer because it minimizes the total number of pairs ${m \choose 2} + {k-m \choose 2}$ 
available in the two subpopulations, hampering coalescence. Lemma \ref{lem:deterministic_balancing_lemma} implies a 
suitable upper bound
\begin{equation}
    X_e \leq_{st}  Z_{\lfloor k_e/2 \rfloor} + Z_{\lceil k_e/2 \rceil}
    \label{eqn:local-balancing-Xe-upper-bound}
\end{equation}
 This upper bound~\eqref{eqn:local-balancing-Xe-upper-bound} motivates applying the same balancing principle recursively throughout the subtree below $e$.

\subsubsection{From Local Balancing to Balanced Trees}
\label{sec:recursive-balancing}

By iterating the ideas from the previous section, it is natural to believe the topology
minimizing coalescence is the one where \textit{all} splits are even, namely when the 
topology below edge $e$ is the balanced tree. This is true in the following sense.

\begin{lemma}
     Given a binary tree $\cT$ (with branch lengths) let $X_\cT$ denote the number of lineages that have not coalesced by the time they reach the root. If $\cB_k$ is the balanced tree with $k$ leaves and all branch lengths equal to $T_{\mathrm{min}}$ then $X_{\cB_k} \geq_{st} X_\cT$ for any other tree $\cT$ with $k$ leaves and minimum branch length $T_{\mathrm{min}}$. 
    \label{lem:balanced_is_worse_case}
\end{lemma}
\noindent The proof is postponed to Appendix~\ref{app:balanced-worst-case-proof}, reducing 
the result to our \textit{balancing lemma} \ref{lem:balancing_lemma}. This extension of
Lemma \ref{lem:deterministic_balancing_lemma} allows for random 
subpopulation sizes, and essentially by iteratively applying it we can transform any tree
$\cT$ to the corresponding balanced tree $\cB_k$.

Lemma \ref{lem:balanced_is_worse_case} suggests a natural way of upper-bounding 
the number of lineages $X_e$ entering a given edge $e$: $X_e \leq_{st} X_{\cB_{k_e}}$. From this, we get our final bound.

\begin{theorem}[Balanced Bipartition Cover Bound]
If we have a species tree with $k$ species and a set $\cG$ of at least
\begin{equation}
    M_b(k, T_{\mathrm{min}}):=  \inf\left\{n \in \N : \sum_{\ell=2}^{k-2}\left[1-w_\ell\right]^n \leq 1-q\right\}
\label{eqn:balanced-bound}
\end{equation}
\noindent gene trees sampled independently under the MSC model, where

\begin{equation}
    w_\ell := \P(W_{\ell} = 1)
    \label{eqn:def_of_zl}
\end{equation}

\noindent and where the distributions of $W_\ell := Z_{X_{\cB_\ell}}$ are defined recursively by $W_1 \equiv 1$ and
\begin{equation}
\P(W_\ell=j)
=
\sum_{a\geq 1}\sum_{b\geq 1}
\P\left(W_{\lfloor\ell/2\rfloor}=a\right)
\P\left(W_{\lceil\ell/2\rceil}=b\right)
g_{a+b,j}(T_{\mathrm{min}}),
\qquad \ell\geq 2.
\label{eqn:balanced-W-recursion}
\end{equation}

\noindent Then $\cG$ forms a bipartition cover with probability at least $q$. Here $T_{\mathrm{min}}$ is the minimum branch length (in coalescent units) in the species tree.
\label{thm:improved_bipartition_cover_bound_balanced}
\end{theorem}

\begin{proof}[Reduction to Lemma \ref{lem:balanced_is_worse_case}]
    First, we show the recursion above does produce the distribution of 
    $W_\ell = Z_{X_{\cB_\ell}}$. We do this by induction on $\ell$. The base 
    case of $\ell =1$ is trivial as $W_1 = Z_{X_1} = Z_1 \equiv 1$. For the inductive 
    step, note that the  balanced tree with $\ell$ leaves has balanced 
    subtrees with $\lceil \ell/2 \rceil$ and $\lfloor \ell/2\rfloor$ leaves respectively.
     By induction, $W_{\lceil \ell/2 \rceil}$ and $W_{\lfloor \ell/2 \rfloor}$ give the distributions 
     of the numbers of lineages exiting these subtrees, respectively. These counts are independent under the MSC.
     Conditional on their values being $a$ and $b$, the resulting $a+b$ lineages traverse one more branch of length
     $T_{\mathrm{min}}$, giving Equation~\eqref{eqn:balanced-W-recursion}. 

    Lemma \ref{lem:balanced_is_worse_case} implies $X_e \leq_{st} X_{\cB_{k_e}}$ 
    so by the remark below Lemma \ref{lem:bounding_prob_with_second_order_stochastic_dominance} it 
    suffices to show that $U_\ell := X_{\cB_\ell}$ is stochastically increasing in $\ell$. This follows from
    induction as
        \[U_\ell \overset{d}{=} Z_{U_{\lfloor \ell/2 \rfloor}} + Z'_{U_{\lceil \ell/2 \rceil}}.\]
    \noindent When $\ell$ increases by one, exactly one of $\lfloor \ell/2\rfloor$ and
    $\lceil \ell/2\rceil$ increases by one. Hence the inductive hypothesis and
    Corollary~\ref{corollary:stochastic_dominance_composition} imply
    $U_{\ell+1} \geq_{st} U_\ell$. It follows that
    \[
        h(\ell):=\E g_{U_\ell,1}(T_{\mathrm{min}})
        =\P(W_\ell=1)=w_\ell
    \]
    is decreasing. Applying Theorem~\ref{thm:caterpillar_bound} with this choice of
    $h$ gives the result.
\end{proof}

\noindent As the subtrees of a balanced tree are balanced, we can calculate $w_\ell$ in a recursive fashion. In our asymptotic analysis below, 
we will discuss the following convenient upper bound
\[M_b(k,T) \leq \frac{\log\left(\frac{k-3}{1-q}\right)}{-\log\left(1-g_{\E X_{\cB_{k-2}}, 1}(T)\right)}.\]
\noindent Hence our analysis has allowed us to replace $k-2$ by $\E X_{\cB_{k-2}}$, which depending on
 $T$ may be a large improvement.

\section{Simulations}
\label{sec:simulations}

In this section we compare the performance of our new bounds both to the original bound of \citet{uricchio2016_bipartition_cover} and to empirical results coming from simulations of the MSC model under various species tree topologies. All relevant code can be found on GitHub \citep{github_page}.

\subsection{Bound Growth Rates}
\label{sec:bound-growth-rates}

First, we empirically explored how all the bounds grow as functions of the two key parameters: the minimum branch length $T_{\mathrm{min}}$ and the species count $k$.  Figure \ref{fig:bound_comparison} highlights our results. We can observe that all bounds increase with $k$ and decrease with $T_{\mathrm{min}}$ as is natural. However, our newer bounds are dramatically lower in essentially all regimes.

While it varies species to species, the maximal number of independent genes is typically on the order of $10^3$ to $10^5$ (\cite{Loman2012BacterialGenomes}, \cite{Pertea2023HumanGeneCatalogue}). One drawback of the original bipartition cover bound \ref{thm:uricchios_bipartition_cover_bound} that Figure \ref{fig:bound_comparison} highlights is that it dips above this threshold even for relatively few species $k$, especially when the minimum branch length is short. This can limit its practicability because scientists might not have access to enough genes to achieve their desired coverage probability.

Thus, one attractive feature of our new bound is that it remains below these biologically reasonable thresholds for a much wider choice of the relevant parameters $k, T_{\mathrm{min}}$, greatly improving its utility.

\subsection{Improvement Ratios}
\label{sec:improvement-ratios}

To quantify how much these new bounds improve on the bound of \cite{uricchio2016_bipartition_cover} and over each other, we plot the ratios $I := M_{\mathrm{old}}/M_{\mathrm{new}}$ again as functions of the two nontrivial parameters $T_{\mathrm{min}}, k$. Values of $I > 1$ indicate an improvement over the original bound, and values $I < 1$ a degradation. Figure \ref{fig:improvement_ratios} compares the various bounds we have constructed throughout this paper. 

As we discussed above, the balanced bound improves on the caterpillar bound, which improves on the original bound of \citet{uricchio2016_bipartition_cover}. As such, all improvement ratios $I \geq 1$. But as Figure \ref{fig:improvement_ratios} highlights, most of the improvement comes from the more sophisticated balanced bound (Theorem \ref{thm:improved_bipartition_cover_bound_balanced}), which improves on the original bound by several orders of magnitude in the challenging high $k$ and low $T_{\mathrm{min}}$ regimes. 

On the other hand, the caterpillar bound only improves by a small constant factor. As we discuss in Section \ref{sec:estimated_growth_rates} below, this is likely due to the fact that most terms in the sum $\sum_{\ell=2}^{k-2} (1-g_{\ell,1}(T))^n$ quickly saturate for large $\ell$, so we gain little by replacing the maximum $g_{k-2,1}(T)$ by $g_{\ell, 1}(T)$.

\subsection{Quantifying Overestimation}
\label{sec:quantifying-overestimation}

In this section, we aim to quantify the level to which our bounds overestimate the number of gene trees required to obtain a cover. Namely, if $n_e$ is the true $q$-quantile and $n_b$ is one of our upper bounds, define the \textit{overestimation ratio} to be the quantity $n_b/n_e$. To estimate $n_e$, we empirically generate gene trees from the MSC model until we get a bipartition cover, and repeat this for $N$ independent trials to get an estimate of the desired quantile. In this study, we chose $N = 10^4$ trials and $q = 0.9$.

We explore a few different species tree topologies: the caterpillar tree with equal branch lengths, the balanced tree with equal branch lengths, and trees generated randomly from the Yule model. For Yule trees we normalize the branch lengths so they have the desired minimum length $T_{\mathrm{min}}$. We sample $100$ different Yule trees in this way in order to study the distribution of overestimation ratios. 

Since caterpillar and balanced trees represent the two worst cases we explored in this paper, they in a sense measure how close our bound is to optimal in a topology-free sense. We expect our bound to be tighter in these scenarios as compared to the more average case of Yule trees. Our results for caterpillar and balanced trees are displayed in Figures \ref{fig:empirical_overestimation_ratios_for_caterpillar_trees} and \ref{fig:empirical_overestimation_ratios_for_balanced_trees} respectively. The overestimation ratio is significantly higher for balanced trees as compared to caterpillars. Since our bound is topology-free, this aligns with the empirical observation that caterpillar trees are difficult to recover under the MSC. Overall, the new bound is still fairly far from tight, suggesting incorporating even partial topological information might be necessary.

Examining its dependence on the model parameters, we observe that the overestimation ratio of our new bound varies less with the species parameter $k$ as compared to the original bound, indicating that the balanced bound should scale more favorably to larger values of $k$ compared to the original bound. On the other hand, performance appears to deteriorate significantly as $T_{\mathrm{min}}$ decreases, suggesting less favorable scaling in the small $T_{\mathrm{min}}$ regime.

Since most trees differ substantially from these worst-case configurations, the performance on Yule trees provides a more realistic measure of the typical behavior of the bound. To assess this, we sample a range of species tree topologies from the Yule model and compute the corresponding overestimation ratios. Figure \ref{fig:empirical_overestimation_ratios_for_yule_trees} shows the resulting distribution of these ratios across various choices of $k$ and $T_{\min}$. 

As expected, performance in this setting is noticeably worse than in the structured worst-case scenarios, but it remains substantially better than that of the original bound. Increasing the species parameter $k$ appears to have a much larger impact than in the previous settings, suggesting worse scaling. This highlights that incorporating even partial topological information may be essential for achieving further improvements.

\subsection{Estimated Growth Rates}
\label{sec:estimated_growth_rates}

 Since the bounds are complicated functions of $T_{\mathrm{min}}$ and $k$, we develop estimates of their scaling with respect to these two parameters.

 Here we confirm our intuition that our new bounds drastically improve performance in the small $T$ regime. For fixed $T$, we show that all the bounds are $\Theta_T(\log(k))$, and that this is essentially the best one can do while relying on the union bound. We start with a result that specifies the asymptotics of the function $g_{k,1}(T)$ underlying all our bounds.

\begin{lemma}[Asymptotics of $g_{k,1}(T)$]
The function $T \mapsto g_{k,1}(T)$ has asymptotics
\begin{equation}
     \begin{cases} 1 - g_{k,1}(T) \sim \frac{3(k-1)}{k+1} \exp(-T), & T \uparrow \infty, \\ 
 g_{k,1}(T) \sim \frac{k!}{2^{k-1}} T_k^{k-1}, &  k^2T_k = o(1).
\end{cases}
\label{eqn:asymptotics_T_mapsto_g_k1(T)}
\end{equation}
\noindent Moreover, for fixed $T$, we have $g_{k,1}(T) \downarrow \P(S_\infty \leq T) =: s(T)$ where $S_\infty := \sum_{i=2}^\infty \tau_i$ for $\tau_i \sim \mathrm{Exp}\left({i \choose 2}\right)$ independent. Furthermore, the gap $\Delta_k(T) := g_{k,1}(T) - s(T)$ satisfies
\[\Delta_k(T) = \frac{2f_{S_\infty}(T)}{k} + o\left(\frac{1}{k}\right) \sim \frac{2f_{S_\infty}(T)}{k}.\]
\label{lem:asymptotics_g_k1(T)}
\end{lemma} 

\noindent The proof can be found in Appendix~\ref{app:coalescent-asymptotics-proof}. 
Using this, we develop the following asymptotics for our two bounds. Recall that $M_o(k,T)$ denotes the original bound of \citet{uricchio2016_bipartition_cover} and $M_b(k,T)$ our balanced bound (Theorem~\ref{thm:improved_bipartition_cover_bound_balanced}). 

\begin{lemma}[Bound Asymptotics]
    Let $\kappa_{q,k} := \log((k-3)/(1-q))$ and $u(T) := \frac{2-\exp(-T/2)}{1-\exp(-T/2)}$. Then
    \[M_o(k,T)  \sim \begin{cases}
    \frac{\kappa_{q,k}}{T}, & T \uparrow \infty, \\
    \frac{2^{k-3}\kappa_{q,k}}{(k-2)! T^{k-3}}, & T = o(k^{-2}), \\
    \frac{\kappa_{q,k}}{-\log(1-s(T))}, & T \text{ fixed}, k \to \infty
\end{cases}\]

    \[M_b(k,T) \lesssim 
        \frac{\kappa_{q,k}}{-\log(1-g_{u(T), 1}(T))}, \qquad T \text{ fixed}, k \to \infty\]

    \noindent Moreover, in the fixed $T$ regime, the improvement factor satisfies
    \[
    \liminf_{k\to\infty}\frac{M_o(k,T)}{M_b(k,T)}
    \geq \beta_T
    :=\frac{\log(1-g_{u(T),1}(T))}{\log(1-s(T))},
    \qquad
    \beta_T\sim \frac{\pi^2}{2T}
    \quad\text{as }T\downarrow0.
    \]
    \label{lem:bound_asymptotics}
\end{lemma}

\noindent The proof can be found in Appendix~\ref{app:bound-asymptotics}.

To see why the $\Theta(\log(k))$ growth rate is natural for the $k$ asymptotics, note that $g_{k,1}(T)$ is roughly constant and equal to $s(T) > 0$ for all large $k$. Hence in large caterpillar trees, where most of the edges have large descendant counts, most of their bipartitions have $g_{k,1}(T)$ saturating in this way. 

In this setting, we are essentially saying each bipartition corresponds to a $\mathrm{Geometric}(s(T))$ random variable which specifies how many gene trees we need for that specific bipartition to show up. Each such variable has tail $\P(\mathrm{Geo}(s) > m) = (1-s)^{m}$. Solving for $p(m) = k^{-1}$ we see $m = \log(k)/[-\log(1-s)]$, so the maximum of $k$ independent such geometric random variables ought to be roughly of this order.

\section{Conclusion}
\label{sec:conclusion}

We developed topology-free improvements to the bipartition cover bound of
\citet{uricchio2016_bipartition_cover} by identifying two distinct extremal
features of species trees. Caterpillar trees maximize the descendant-count
profiles that enter the union-bound argument, whereas balanced trees
stochastically maximize the number of lineages that remain uncoalesced for a
fixed number of descendants. Combining these observations yields a recursively
computable balanced bound that can improve on the original bound by several
orders of magnitude across the parameter regimes considered in our simulations.

In analyzing coalescence under the MSC, our key technical tools are the balancing lemmas, Lemmas \ref{lem:deterministic_balancing_lemma} and
\ref{lem:balancing_lemma}. These formalize the intuition that, when a fixed collection of 
lineages is distributed more evenly across descendant populations, more lineages tend to remain 
uncoalesced. While short internal branches and anomaly-zone behavior are familiar obstacles to 
species-tree inference, our analysis identifies a distinct topology-dependent mechanism: balanced 
splits systematically delay coalescence, allowing ancestral lineages to persist longer 
and thereby increasing the potential for incomplete lineage sorting.

For fixed branch length, the new bound grows as $\Theta_T(\log k)$, which is the
best possible order within the union-bound framework used here. The bound need
not be globally tight, however, because its two extremal ingredients are attained
by different tree topologies. Our simulations likewise show that some
conservativeness remains, particularly for typical rather than extremal trees.
Incorporating partial information about the species tree topology therefore
provides a natural direction for obtaining sharper finite-sample guarantees.


\newpage
\begin{center}
    \captionsetup{type=figure, justification=raggedright, singlelinecheck=false}
    \includegraphics[height=5cm]{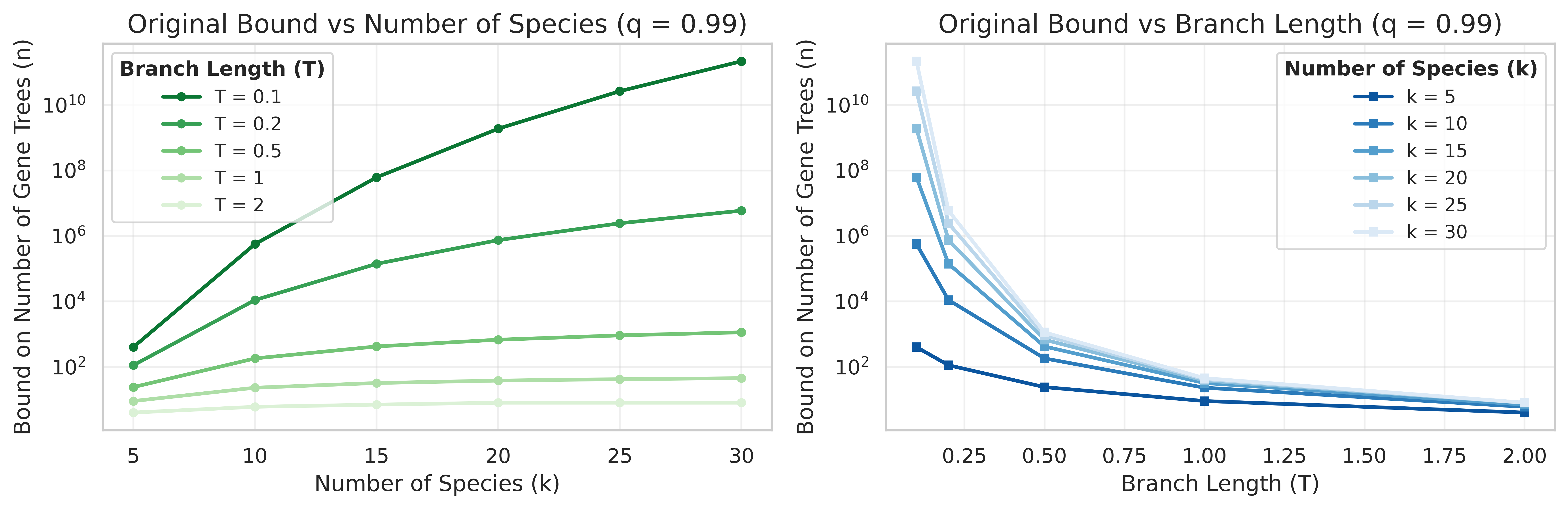}
    
    \includegraphics[height=5cm]{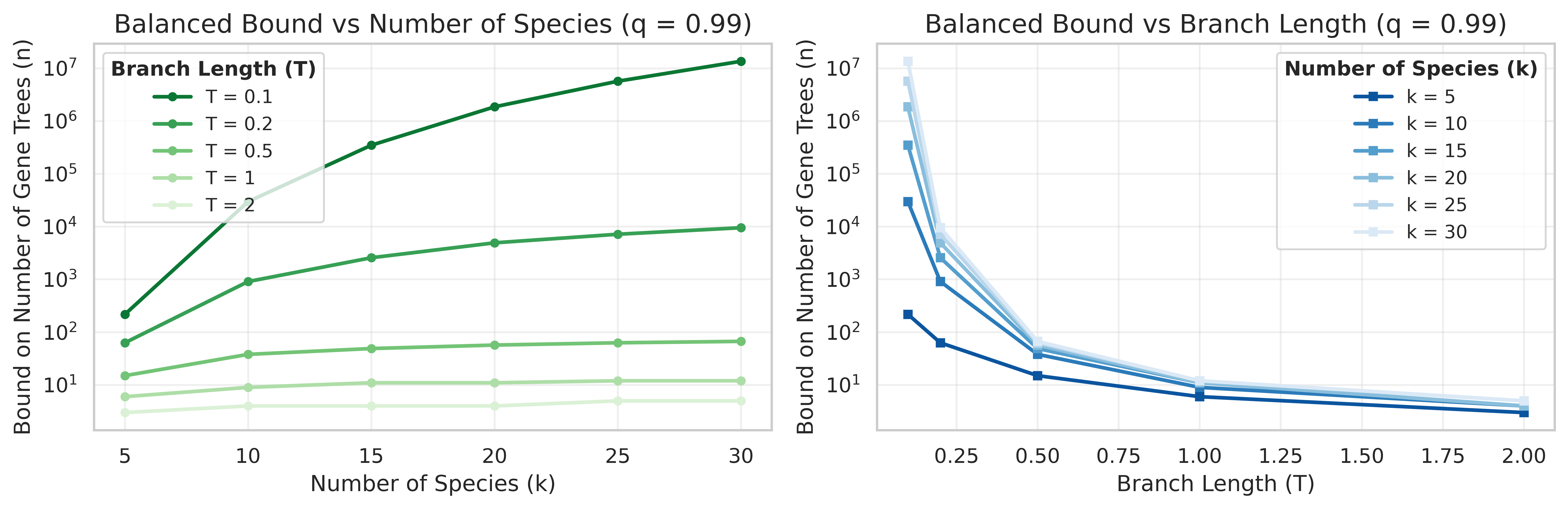}
    \captionof{figure}{Bipartition cover bounds as functions of the two key parameters: $k$, $T_{\min}$. (top) The original bound \citep{uricchio2016_bipartition_cover}. (bottom) Our balanced bound (Theorem~\ref{thm:improved_bipartition_cover_bound_balanced})}
    \label{fig:bound_comparison}
\end{center}

\begin{center}
    \captionsetup{type=figure, justification=raggedright, singlelinecheck=false}
    \includegraphics[height=5cm]{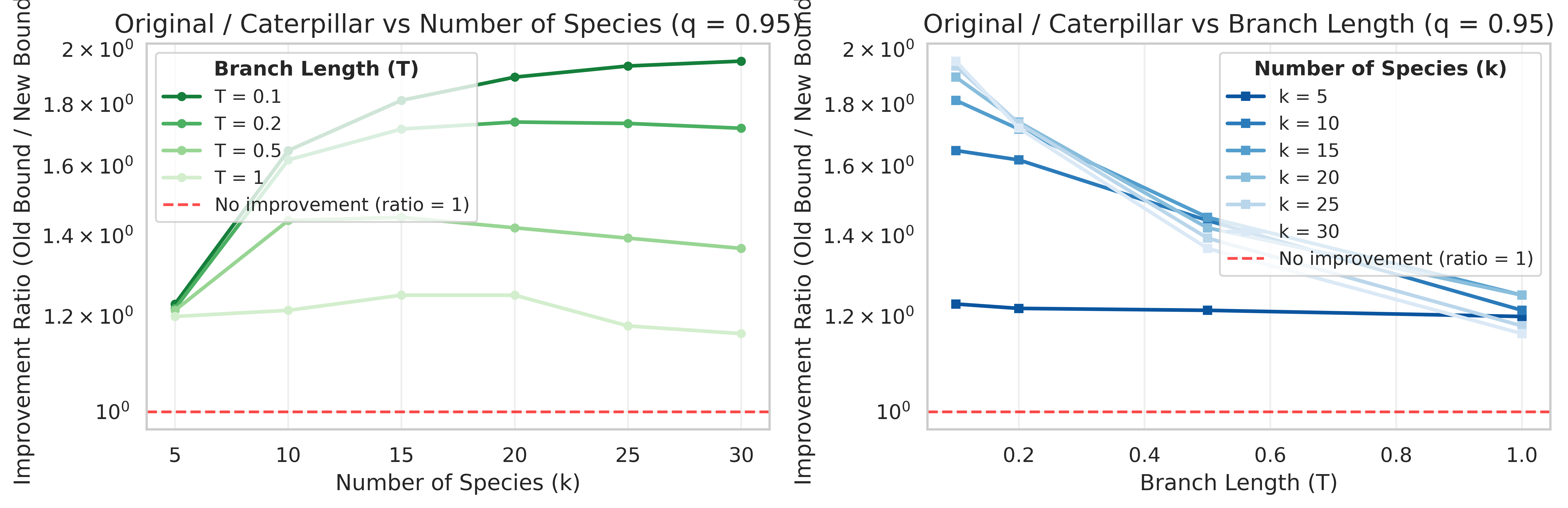}
    
    \includegraphics[height=5cm]{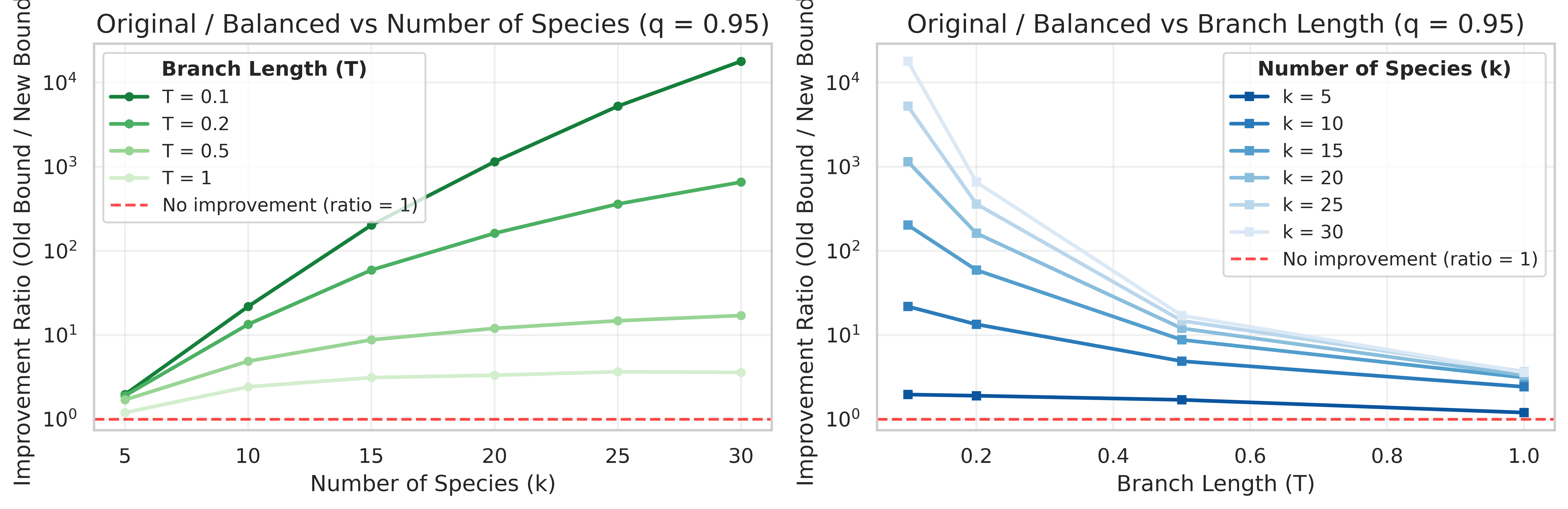}
    \captionof{figure}{Improvement ratios. (top) Improvement of caterpillar bound (Theorem~\ref{thm:caterpillar_bound}) over the original bound (Theorem~\ref{thm:uricchios_bipartition_cover_bound}). (bottom) Improvement of the balanced bound over original}
    \label{fig:improvement_ratios}
\end{center}

\newpage
\begin{center}
    \captionsetup{type=figure, justification=raggedright, singlelinecheck=false}
    \includegraphics[height=5cm]{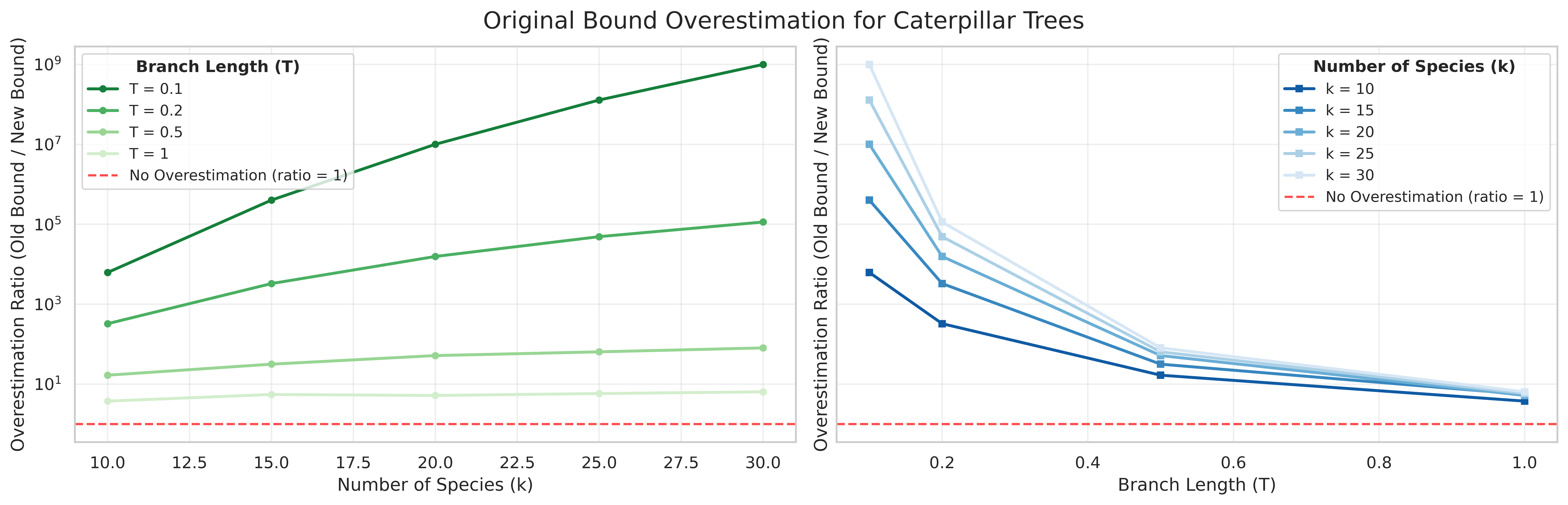}
    
    \includegraphics[height=5cm]{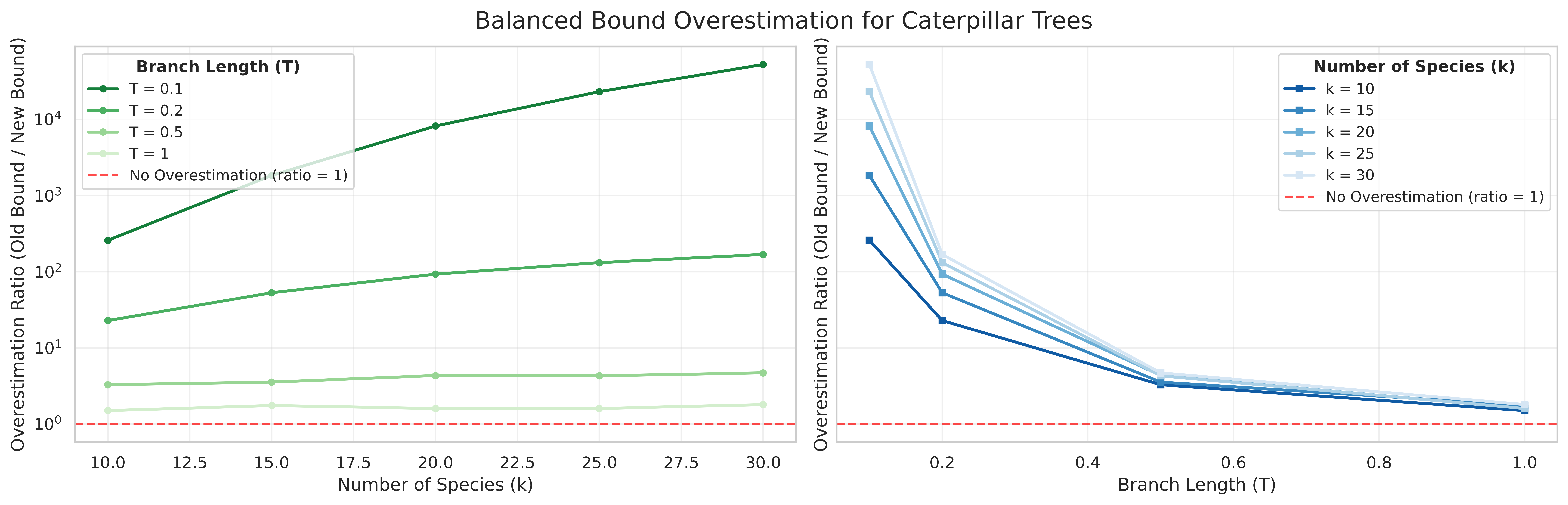}
    \captionof{figure}{Empirical overestimation ratios for caterpillar trees. (top) The original bound (Theorem~\ref{thm:uricchios_bipartition_cover_bound}) and (bottom) our balanced bound (Theorem~\ref{thm:improved_bipartition_cover_bound_balanced})}
    \label{fig:empirical_overestimation_ratios_for_caterpillar_trees}
\end{center}

\begin{center}
    \captionsetup{type=figure, justification=raggedright, singlelinecheck=false}
    \includegraphics[height=5cm]{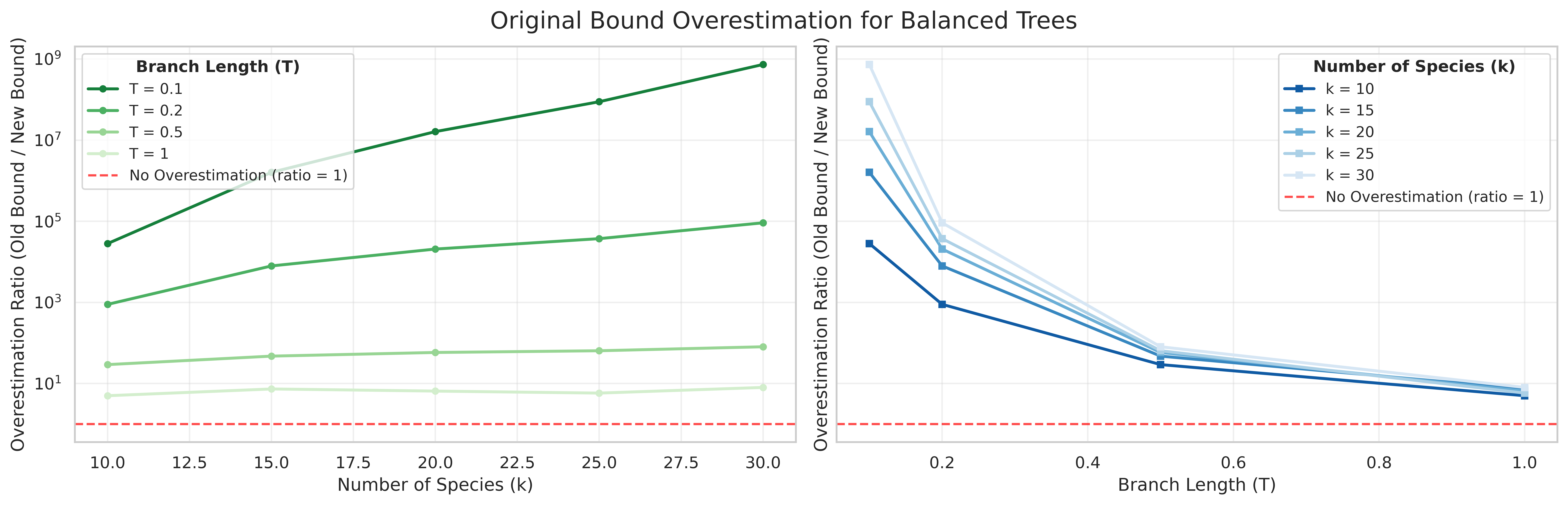}
    
    \includegraphics[height=5cm]{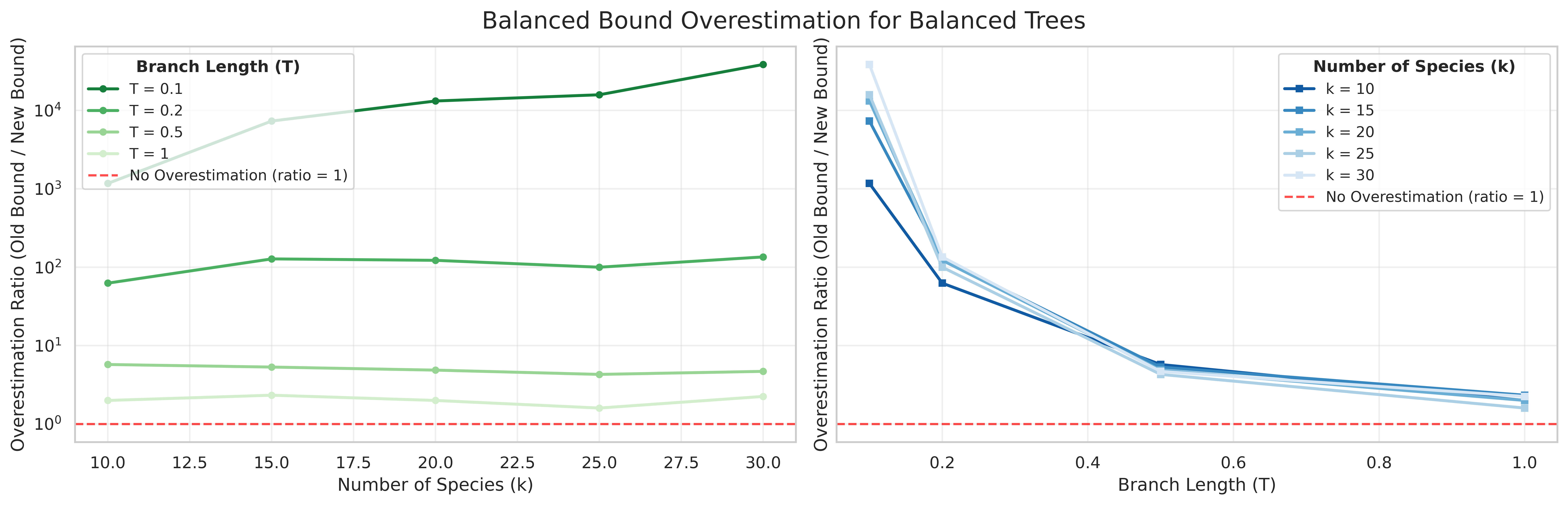}
    \captionof{figure}{Empirical overestimation ratios for balanced trees. (top) The original bound (Theorem~\ref{thm:uricchios_bipartition_cover_bound}) and (bottom) our balanced bound (Theorem~\ref{thm:improved_bipartition_cover_bound_balanced})}
    \label{fig:empirical_overestimation_ratios_for_balanced_trees}
\end{center}

\begin{figure}[htbp]
    \centering
    \captionsetup{justification=raggedright,singlelinecheck=false}
    \includegraphics[height=0.85\textheight,keepaspectratio]{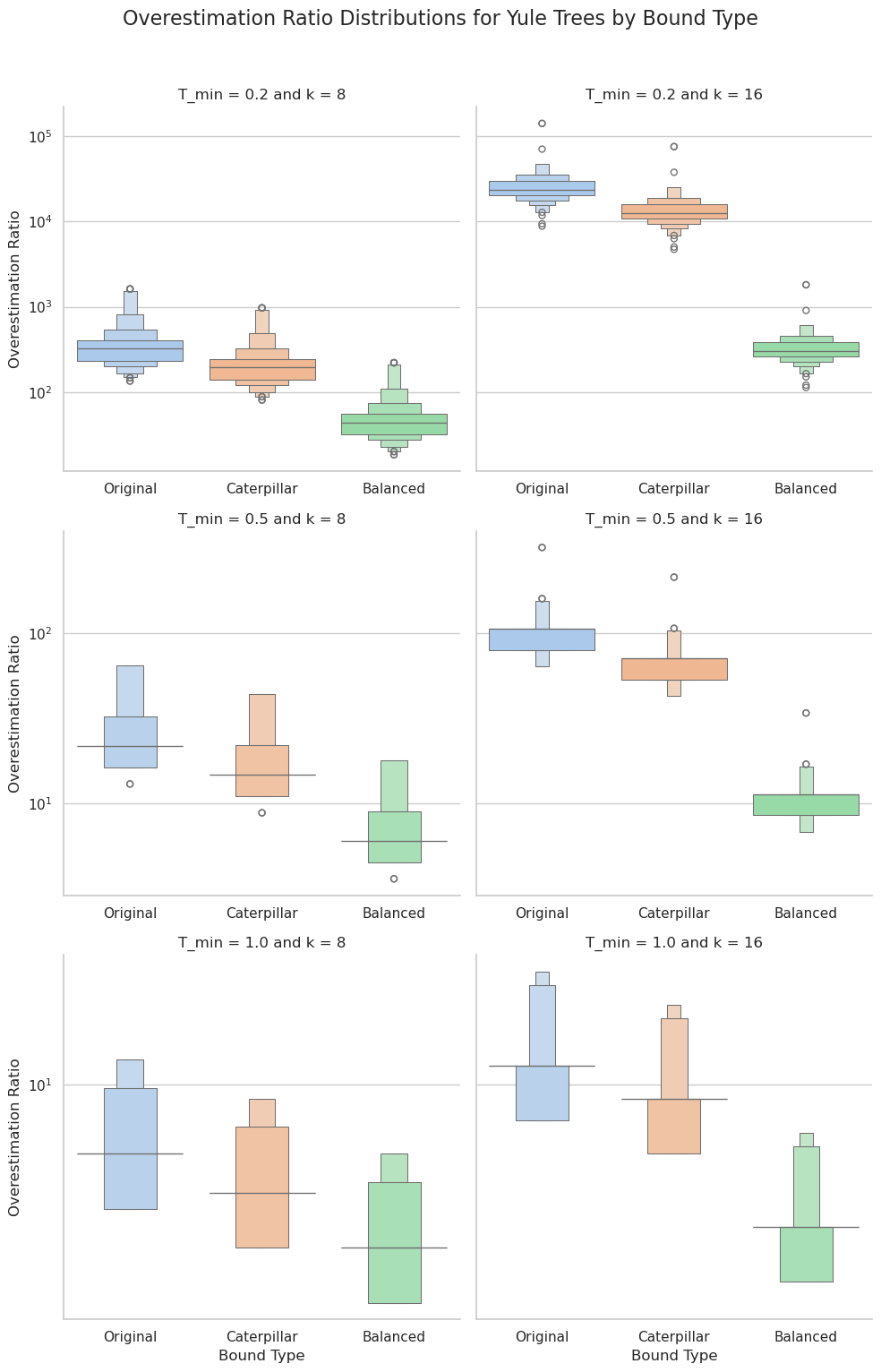}

    \caption{Boxenplots for the distribution of (log) overestimation ratios for Yule trees across different choices of our parameters $T_{\mathrm{min}}, k$. Each row corresponds to a specific choice of $T_{\mathrm{min}} \in \{0.2, 0.5, 1\}$ and each column a choice of $k \in \{8, 16\}$}
    \label{fig:empirical_overestimation_ratios_for_yule_trees}
\end{figure}

\newpage
\appendix
\numberwithin{figure}{section}
\numberwithin{table}{section}
\section{Basic Results}
\label{app:basic-results}

    \subsection{Basic Properties of Log-Concavity}
    \label{app:basic-log-concavity}

    A sequence $\{x_k\}$ is \textit{log-concave} if $x_{k}^2 \geq x_{k-1}x_{k+1}$. We say a random variable $X$ is log-concave if the sequence $x_k = \P(X = k)$ is log-concave. 
    
    In order to apply Theorem \ref{thm:convolutions_and_likelihood_ratio_order}, we will need to determine whether various variables of interest are log-concave. For this purpose, we introduce a few basic closure properties of log-concave random variables. It will turn out to be useful to work with a strengthening of log-concavity. We say a sequence $\{x_k\}$ is \textit{ultra log-concave (ULC)} if the sequence $\{k! \cdot x_k\}$ is log-concave, or equivalently $kx_k^2 \geq (k+1)x_{k-1}x_{k+1}$. Clearly ultra log-concavity implies log-concavity. The following result, and much more, can be found in Chapter 4 of \cite{saumard2014logconcavitystronglogconcavityreview}.

    \begin{theorem}[Pr\'ekopa]
        If $X, Y$ are independent, (ultra) log-concave, integer-valued random variables, then $X+Y$ is (ultra) log-concave. Equivalently, the class of (ultra) log-concave distributions is closed under convolution. 
        \label{thm:prekopas_theorem}
    \end{theorem}
    
    \noindent The following shows us that Kingman's coalescent preserves log-concavity in our setting of interest.

\begin{theorem}[Preservation of Ultra Log-Concavity]
\noindent Suppose $P_t$ is the transition semigroup of a pure-death process with death rates $\lambda_i$ so that (a) $\lambda_i$ is convex (b) $\eta_i := \lambda_i/i$ is concave and (c) at least one of these is strictly convex/concave. Moreover, suppose that $\pi$ is ultra log-concave with either $\mathrm{supp}(\pi) = \{2,..., M\}$ or $\mathrm{supp}(\pi) = \{1,..., M\}$ for some $M \in \N \cup \{\infty\}$. Then $\pi(t) := \pi P_t$ is ultra log-concave for all $t \geq 0$. In particular, this is true for Kingman's coalescent with $\lambda_i = {i \choose 2}$. 
\label{thm:kingman_preserves_ultralogconcave}
\end{theorem}

\begin{proof}[Proof of Theorem \ref{thm:kingman_preserves_ultralogconcave}]
     Recall Kolmogorov's forward equations imply
    \[\frac{d}{dt} \pi_i(t) = \lambda_{i+1}\pi_{i+1}(t) - \lambda_i \pi_i(t).\]
    \noindent First, we handle the interior. For $i \in \mathrm{Int}(\mathrm{supp}(\pi))$ we can define the ratios
    \[v_i := \lambda_i \cdot \frac{\pi_i}{\pi_{i-1}}, \qquad r_i := \frac{i\pi_i^2}{(i+1)\pi_{i-1}\pi_{i+1}}  = \frac{v_i}{v_{i+1}} \cdot \frac{\eta_{i+1}}{\eta_i}.\]
    \noindent From here, using the forward equations we can see
    \begin{align*}
        \frac{\dot{r_i}}{r_i} & = \frac{d}{dt} \log(r_i)\\
        & = 2 \frac{\dot{\pi_i}}{\pi_i} - \frac{\dot{\pi}_{i-1}}{\pi_{i-1}} - \frac{\dot{\pi}_{i+1}}{\pi_{i+1}}\\ 
        & = \left(-2\lambda_i + \lambda_{i-1} + \lambda_{i+1}\right) + 2\lambda_{i+1}\frac{\pi_{i+1}}{\pi_i} - \lambda_i\frac{\pi_i}{\pi_{i-1}} - \lambda_{i+2} \frac{\pi_{i+2}}{\pi_{i+1}}\\
        & = \left(\lambda_{i-1} -2\lambda_i + \lambda_{i+1}\right) + 2v_{i+1} - v_i - v_{i+2}.
    \end{align*}
    \noindent By assumption, $r_i(0) \geq 1$ for all $i$. Since $r_i(t)$ is continuous in $t$, the only way for $r_i(t) < 1$ to occur is if $r_i(s) = 1$ for some $0 \leq s < t$. We will show at this time that $\dot{r}_i(s) > 0$, and hence we can never drop below one. Suppose $r_i(s) = 1$ and $r_k(s) \geq 1$ for all other $k$. From these, we see respectively
    \[v_i = \frac{\eta_i}{\eta_{i+1}} \cdot v_{i+1}, \qquad v_{i+2} \leq \frac{\eta_{i+2}}{\eta_{i+1}} \cdot v_{i+1}.\]
    \noindent Using these two bounds, we can see
    \[\frac{\dot{r}_i}{r_i} \geq (\lambda_{i-1} - 2\lambda_i + \lambda_{i+1}) - \left[\eta_{i} - 2\eta_{i+1} + \eta_{i+2}\right] \cdot \frac{v_{i+1}}{\eta_{i+1}}.\]
    \noindent Since $\lambda_i$ is convex, the first term is nonnegative. Since $\eta_i$ is concave, the second term is nonnegative. Moreover, by assumption one of these must be strictly positive, and thus $\dot{r}_i/r_i > 0$ as desired. 
    
    Now, we handle the boundaries. In the case $M < \infty$, we trivially have 
        \[M\pi^2_M(t) \geq (M+1)\pi_{M-1}(t)\pi_{M+1}(t) = 0\] 
    \noindent as $\pi_{M+1}(t) = 0$ for all $t$ -- in a pure-death process on $\N$ the support can never increase. For the lower boundary, first consider the case $\mathrm{supp}(\pi) = \{1,..., M\}$. Then again we trivially have $\pi_1(t) \geq 2\pi_0(t)\pi_2(t) = 0$ for all $t$ as we can never drop below one in a pure-death process on $\N$. In the case $\mathrm{supp}(\pi) = \{2,..., M\}$ note that $2\pi_2(0) > 3\pi_1(0)\pi_3(0) = 0$. By continuity, we must have $2\pi_2(t) > 3\pi_1(t) \pi_3(t)$ for small enough $t > 0$. For such a $t$, we see that $\pi(t)$ is ultra log-concave. Moreover, since $\pi_1(t) > 0$ for any positive $t$ we see $\mathrm{supp}(\pi(t)) = \{1,..., M\}$ so we can now apply the previous case to show ultra log-concavity is preserved for all further times.
\end{proof}

\subsection{Basic Properties of $g_{i,j}(T)$ and $Z^T_i$}
\label{app:basic-coalescent-properties}

We will start by proving some basic properties about the function $g_{i,j}(T)$ and the associated random variables $Z_i^T$ with distribution $\P(Z_i^T = j) = g_{i,j}(T)$. Recall that $Z^T_i$ gives the distribution of the number of lineages remaining after starting with $i$ lineages and running Kingman's coalescent for $T$ seconds.

Moreover, recall $X_{\cB_k}$ denotes the number of lineages reaching the root of a balanced tree with $k$ leaves and branch lengths all equal to $T$. For convenience, we denote $X_k := X_{\cB_k}$. Using the above, we will show that all the variables of interest in this paper are log-concave.

\begin{lemma}
The random variables $X_m, Z_{X_m}$ are ultra log-concave, and hence log-concave.
\label{lem:xm_zxm_are_ultra_log_concave}
\end{lemma}

\begin{proof}
    We will prove this by induction on $m$. Trivially $X_1 = Z_1 \equiv 1$ are ultra log-concave. Now suppose for induction that $X_i, Z_{X_i}$ are log-concave for $i < m$. Note
    \[X_m = Z_{X_{\lfloor m/2 \rfloor}} + Z_{X_{\lceil m/2\rceil}}.\]
    \noindent By Pr\'ekopa's Theorem \ref{thm:prekopas_theorem}, the sum of independent ultra log-concave distributions is ultra log-concave, so $X_m$ is ultra log-concave. Since $\mathrm{supp}(X_m) = \{2,..., m\}$, by Theorem \ref{thm:kingman_preserves_ultralogconcave}, Kingman's coalescent preserves ultra log-concavity, so $Z_{X_m}$ is ultra log-concave. This completes the induction. 
\end{proof}

\noindent Now we show our variables $Z_i^T$ have some simple monotonicity properties.

\begin{lemma}
    Let $i \leq j$ and $t \leq T$. Then $Z_i^T \leq_{st} Z_j^T$ and $Z_{i}^T \leq_{st} Z_i^t$.
    \label{lem:cdf_of_Zi}
\end{lemma}

\noindent Heuristically this just states the obvious: the more lineages we start with at the beginning of a branch and the less time we have to coalesce, the more lineages we tend to end with. 

\begin{proof}[Proof of Lemma \ref{lem:cdf_of_Zi}] Again the fact the function is decreasing in $T$ follows immediately via a coupling argument. To see that it is increasing in $i$, we just condition on the first coalescence event. Assume $z \leq i$, otherwise both $\P(Z_i^T \geq z) = \P(Z_{i-1} \geq z) = 0$. Note if $\tau \sim \exp\left({i \choose 2}\right)$ is the time of the first coalescent event

\[\P(Z_i^T \geq z) = \int_0^\infty \P(Z_{i-1}^{T-t} \geq z) \tau(dt) \geq \int_0^\infty \P(Z_{i-1}^T \geq z) \; \tau(dt) = \P(Z_{i-1}^T \geq z)\]

\noindent where we let $\P(Z^{T}_{i-1} \geq z) = 1$ if $T < 0$. This gives $Z_i^T \geq Z_{i-1}^T$ stochastically as desired.
\end{proof}

\noindent Note that $g_{i1}(T) = 1 - \P(Z_i^T \geq 2)$, from which we get the immediate corollary:

\begin{corollary}
    The function $g_{i1}(T)$ is decreasing in $i$ and increasing in $T$.
    \label{corr:g_i1_is_decreasing}
\end{corollary}

 Thus $g_{i,j}(T)$ has some natural monotonicity properties. The following Lemmas show that most of the variability in $g_{i,j}(T)$ occurs when both $T$ and $i$ are small. 

    \begin{lemma}
        For fixed $n,k$, the function $\P(Z^T_n = k) = g_{n,k}(T)$ is log-concave in $T$. 
    \label{lem:log_concavity_gi(T)_in_T}
    \end{lemma}

    \begin{proof}[Proof of Lemma \ref{lem:log_concavity_gi(T)_in_T}]
    We prove this by induction on $n$. For $n=1$ we see $g_{nk}(T) = \ind\{k=1\}$ which is trivially log-concave. For the inductive step, recall that by conditioning on the first coalescent event and using the memoryless property of Kingman's coalescent, we observed
        \[g_{n,k}(T) = \int_0^T g_{n-1, k}(T-t) {n \choose 2} \exp\left(-{n \choose 2}t\right) \; dt. \]
    Namely, we are just convolving the log-concave function $g_{n-1,k}$ with the log concave function $\exp(-{n\choose 2}t)$. By Pr\'ekopa's theorem \ref{thm:prekopas_theorem}, convolutions preserve log-concavity. Thus $g_{n, k}$ is log-concave. 
    \end{proof}

    \begin{lemma}
        For fixed $T$, the function $i \mapsto g_{i,1}(T)$ is convex.
        \label{lem:g_i1_is_convex}
    \end{lemma}

    \begin{proof}[Proof of Lemma \ref{lem:g_i1_is_convex}]
      It is a classical result of Karlin that the transition matrix $P_T := [g_{i,j}(T)]_{i,j\in \N}$ is totally positive (of any order) for any fixed $T$. See for example \cite{karlin1959coincidence}. Theorem 2 of the paper \cite{eppen1966convexity}, again a result of Karlin, then shows such matrices preserve convexity. Since $e_1 = (1,0,0,...)$ is convex and $(g_{i,1}(T)) = P_T e_1$ this implies the sequence $\{g_{i,1}(T)\}_{i \in \N}$ is convex. 
    \end{proof}

    \noindent Next, we study the expected number of lineages that survive under Kingman's coalescent when the number of starting lineages is large. Since the initial coalescent rate ${i\choose 2}$ grows so fast in the number of lineages, eventually adding more lineages has barely any influence on the number remaining at time $T$. This is essentially the well-known observation that Kingman's coalescent can come down from infinity in any positive time. 

    \begin{lemma}[Upper Bound for Expected Number of Lineages]
        For any $i \in N$ and $T > 0$
            \[\E Z_i^T \leq \frac{1}{1-(1-1/i)\exp(-\frac{T}{2})} \leq\frac{1}{1-\exp\left(-\frac{T}{2}\right)}.\]
        In particular, for any fixed $T > 0$ we see $\E Z_i^T = \Theta_T(1)$.
        \label{lem:expected_number_lineages}
    \end{lemma}

    \noindent Note that in the small $T$ regime, where $\exp(-T/2) \approx 1-T/2$, this bound is roughly $2/T$, recovering the well-known asymptotics of Kingman's coalescent as it comes down from infinity. Hence in a sense this bound is roughly tight in that regime. 

    \begin{proof}[Proof of Lemma \ref{lem:expected_number_lineages}]
    Kolmogorov's backwards equations followed by Jensen's inequality implies
        \[\frac{d}{dt} \E Z_i^t \bigg \vert_{t=T} = \E \left(-{ Z_i^T \choose 2} \right) \leq -{ \E Z_i^T \choose 2}.  \]
    \noindent Hence letting $z(t) := \E Z^t_i$ we see $z(t)$ satisfies
        \[\dot{z}(t) \leq -\frac{z(z-1)}{2}, \qquad z(0) = i.\]
    \noindent Let $\dot{y} = - y(y-1)/2$ with $y(0) = i$. Applying separation of variables
    \[y(t) = \frac{1}{1-\left(1-1/i\right) \exp(-t/2)}.\]
    \noindent Since $\dot{z} \leq \dot{y}$ and $z(0) = y(0)$, applying the comparison principle yields
    \[z(t) \leq y(t) \leq \frac{1}{1- \exp(-t/2)},\]
    \noindent as we desired.
    \end{proof}

    \begin{corollary}
        Suppose that $k \in (2^{m-1}, 2^m]$. Then
        \[\E X^T_{\cB_k} \leq \frac{2-\exp(-T/2)}{1-\exp(-T/2) + \exp(-mT/2)/2^m} \leq \frac{2-\exp(-T/2)}{1-\exp(-T/2)} \]
        \label{corr:upperbound_lineages_balanced}
    \end{corollary}

    \begin{proof}[Proof of Corollary \ref{corr:upperbound_lineages_balanced}]
        Fix $T$ and let
        \[u(n) := \frac{1}{1-(1-1/n)\alpha}, \qquad \alpha := \exp(-T/2)\]
        \noindent be from our upper bound in Lemma \ref{lem:expected_number_lineages}. Then $u$ is increasing and concave in $n$. Let $Y_m := X_{\cB_{2^m}}$ and $a_m := \E Y_m$. Clearly by a coupling argument $\E X_{\cB_k} \leq a_m$. Moreover, by concavity
        \[a_{m+1} = \E (Z_{Y_{m}} + Z_{Y_{m}}) \leq 2 \E u(Y_{m}) \leq 2u(a_m)\]
        \noindent In particular, this implies
        \[\frac{1}{a_{m+1}} \geq \frac{1-\alpha}{2} + \frac{\alpha}{2} \cdot \frac{1}{a_m}.\]
        \noindent Iterating this, we see
        \[a_m \leq \frac{2-\alpha}{1-\alpha + (\alpha/2)^m} \leq \frac{2-\alpha}{1-\alpha} \]

    \end{proof}

\subsection{Basic Properties of Stochastic Orderings}
\label{app:basic-stochastic-orderings}

\noindent The following results are just simple general properties of stochastic orderings which we have sometimes reinterpreted in our current setting. The book \cite{shaked2007stochastic_orderings} contains more information on properties of such stochastic orderings, far beyond what we cover here. 

\subsubsection{First-Order Stochastic Dominance}
\label{app:first-order-stochastic-dominance}

In this section, we prove some lemmas relating to first-order stochastic dominance and the random variables $Z^T_i$. Below we typically suppress the $T$ notation unless it is explicitly necessary. The following generalizes Lemma \ref{lem:cdf_of_Zi} by allowing for the number of initial lineages to be random. 

\begin{corollary}[Stochastic Dominance of Compositions]
    If $X \geq_{st} Y$ then
    
    \[Z_X \geq_{st} Z_{Y}.\]
    
    \noindent Moreover, if $X' \geq_{st} Y'$ and $X \indep Y, X' \indep Y'$, or more generally if $(X,Y) \geq_{st} (X',Y')$ coordinate-wise then
        \[Z_X + Z_{Y} \geq_{st} Z_{X'}+ Z_{Y'}.\]
\label{corollary:stochastic_dominance_composition}
\end{corollary}

\noindent In our context, the first statement above just says that if we tend to have more lineages entering a given edge of our tree, then we tend to also have more lineages exiting it. The second statement says that if an edge $e$ tends to have more lineages uncoalesced in both its subtrees, then the edge $e$ tends to have more lineages entering it. 

\begin{proof}[Proof of Corollary \ref{corollary:stochastic_dominance_composition}]
    For the first statement, note that
    \[\P(Z_X \geq t) = \E_{i \sim X} \P(Z_i \geq t) \geq \E_{i \sim Y} \P(Z_i \geq t) = \P(Z_Y \geq t),\]
    \noindent as by Lemma \ref{lem:cdf_of_Zi} the function $i \mapsto \P(Z_i \geq t)$ is increasing. Hence $Z_X \geq Z_Y$ stochastically. The second claim follows immediately by conditioning on $Z_X$
    \begin{align*}
        \mathbb{P}(Z_X + Z_Y \geq t) &= \mathbb{E}_{z \sim Z_X} \mathbb{P}(Z_Y \geq t-z) \\
        &\geq \mathbb{E}_{z \sim Z_X} \mathbb{P}(Z_{Y'} \geq t-z) \\
        &\geq \mathbb{E}_{z \sim Z_{X'}} \mathbb{P}(Z_{Y'} \geq t-z) \\
        &= \mathbb{P}(Z_{X'} + Z_{Y'} \geq t)
    \end{align*}
    \noindent so that $Z_X + Z_Y \geq_{st} Z_{X'} + Z_{Y'}$. The joint statement follows similarly just from the fact $\phi(a,b) := \P(Z_a + Z_b \geq t)$ is trivially increasing coordinate-wise. 
\end{proof}

\noindent The following tells us that increasing mixtures of increasing variables are again increasing. 
\begin{lemma}
    Suppose that we have collections of random variables $\{V_i\}_{i \in \N}, \{\Theta_i\}_{i \in \N}$ so that for each $i \leq j$ we have:
    \begin{enumerate}[label=\alph*)]
        \item $\Theta_i \leq_{st} \Theta_j$,
        \item $V_i \mid \{\Theta_i = \theta\} \leq_{st} V_j \mid \{\Theta_j = \theta\}$ for every fixed $\theta$,
         \item $V_i \mid \{\Theta_i = \theta\} \leq_{st} V_i \mid \{\Theta_i = \theta'\}$ for every $\theta \leq \theta'$ and $i$.
    \end{enumerate}
    \noindent Then $V_i \leq_{st} V_j$ for $i \leq j$. 
\label{lem:mixtures_and_stochastic_dominance}
\end{lemma}

\begin{proof}[Proof of Lemma \ref{lem:mixtures_and_stochastic_dominance}]
    For $i < j$ note by conditioning on $\Theta_i$ we get
    \begin{align*}
        \P(V_i \geq x) & = \int_\R \P(V_i \geq x \; | \; \Theta_i = \theta) \; \Theta_i(d\theta) \\
        & \leq \int_\R \P(V_j \geq x \; | \; \Theta_j = \theta) \; \Theta_i(d\theta)\\
        & \leq \int_\R \P(V_j \geq x \; | \; \Theta_j = \theta) \; \Theta_j(d\theta) \\
        & = \P(V_j \geq x)
    \end{align*}    
    \noindent by first using property (b) and then properties (a),(c) together. For the latter, we are using the fact that (c) implies $\P(V_j \geq x | \Theta_j = \theta)$ is an increasing function in $\theta$, and since $\Theta_i \leq_{st} \Theta_j$ this ordering is preserved when you take the expectation with respect to any increasing function.  
    \end{proof}

    \subsubsection{Likelihood Ratio Ordering}
    \label{app:likelihood-ratio-ordering}
    We say a random variable $X$ is less than $Y$ with respect to the \textit{likelihood ratio ordering (LRO)}, and denote this by $X \leq_{lr} Y$,  if the likelihood ratios $f_Y(t)/f_X(t)$ are non-decreasing in $t$. Here, $f_X, f_Y$ are the mass functions of $X,Y$ respectively. Equivalently, we require
    \[f_X(s)f_Y(t) \geq f_X(t)f_Y(s), \quad \forall s \leq t\]
    \noindent It is not hard to see that $X \leq_{lr} Y$ implies $X \leq_{st} Y$, so this is a strictly stronger ordering than that of first-order stochastic dominance. The LRO has some convenient closure properties. The following appear as Theorems 1.C.9 and 1.C.17 of \cite{shaked2007stochastic_orderings} respectively:

    \begin{theorem}[LRO preserved by Convolutions]
        Suppose we have two collections $\{X_i\}_{i=1}^n, \{Y_i\}_{i=1}^n$ of independent, log-concave random variables. If $X_i \leq_{lr} Y_i$ for all $i$ then $\sum_{i=1}^n X_i \leq_{lr} \sum_{i=1}^n Y_i$.
        \label{thm:convolutions_and_likelihood_ratio_order}
    \end{theorem}

\begin{theorem}[LRO preserved by Mixtures]
Suppose we have a collection $\{V_i\}_{i\in \N}$ of random variables so that $V_i \leq_{lr} V_j$ for $i \leq j$. Moreover, suppose that we have random variables $M,N$ independent of $\{V_i\}$ with $M \leq_{lr} N$. Then $V_M \leq_{lr} V_N$.
\label{thm:mixtures_and_likelihood_ratio_order}
\end{theorem}

\noindent Using these, we can see our random variables of interest are increasing with respect to the LRO:

\begin{lemma}
For $i \leq j$ we have $Z_i \leq_{lr} Z_j$, $X_i \leq_{lr} X_j$ and $Z_{X_i} \leq_{lr} Z_{X_j}$.
\label{lem:xi_zxi_increasing_in_lro}
\end{lemma}

\begin{proof}
    First, we show $\{Z_i\}_{i \in \N}$ is increasing in the LRO, namely

    \[\P(Z_i = m) \cdot \P(Z_j = n) \geq \P(Z_i = n)\cdot \P(Z_j  = m), \quad \forall i \leq j, \; \forall m \leq n. \]

    \noindent But this is a consequence of the total positivity of the transition kernels for birth-death processes (see for example \cite{karlin1959coincidence}).
    
    We will prove the other statements by induction. Trivially $X_1 \equiv 1 \leq_{lr} 2 \equiv X_2$. Thus suppose for induction that $X_1 \leq_{lr} ... \leq_{lr} X_i$. Then Theorem \ref{thm:mixtures_and_likelihood_ratio_order} implies $Z_{X_1} \leq_{lr} ... \leq_{lr} Z_{X_i}$. Using this, Theorem \ref{thm:convolutions_and_likelihood_ratio_order} implies
    \[X_i = Z_{X_{\lfloor i/2 \rfloor}} + Z_{X_{\lceil i/2 \rceil}} \leq_{lr} Z_{X_{\lfloor (i+1)/2 \rfloor}} + Z_{X_{\lceil (i+1)/2 \rceil}} = X_{i+1}\]
    \noindent since $Z_{X_m}$ are log-concave by Lemma \ref{lem:xm_zxm_are_ultra_log_concave}. Hence the result follows from induction. 
    
\end{proof}

\newpage 
\section{Proofs of the Main Upper-Bound Results}
\label{app:main-upper-bound-proofs}

\subsection{Proof of Lemma \ref{lem:caterpillar_maximize_increasing_sums}: Caterpillars Maximize Increasing Sums}
\label{app:caterpillar-maximization-proof}

\noindent The following is true regardless of our convention for choosing the edge adjacent to the root. Moreover, it is clearly also true if we use \textit{all} descendant counts $\{\alpha_i\}_{i=1}^{2k-2}$ rather than just those corresponding to nontrivial bipartitions: all but one of the additional $\alpha_i$ are equal to one, and it is clear from the proof below we can handle the exceptional value as well. \\

\noindent \textbf{Lemma:} Suppose $\{\alpha_{i}\}_{i=1}^{k-3}$ are the descendant counts of our rooted binary tree $T$ corresponding to our nontrivial bipartitions. Moreover, suppose $f: \N \to \R$ is nondecreasing. Then $\sum_i f(\alpha_{i})$ is maximized when $T$ is the caterpillar tree, for which $\{\alpha_{i}\} = \{2,3,..., k-2\}$.

\begin{proof}[Proof Lemma \ref{lem:caterpillar_maximize_increasing_sums}]
Let's start by proving the caterpillar is the maximizer. Since $f$ is nondecreasing, it suffices to show we can reindex $\{\alpha_i\}_{i=1}^{k-3}$ so that $\alpha_i \leq i + 1$. Heuristically, we are just pairing up each edge $e$ in our tree $T$ with an edge $e'$ in our caterpillar tree so that $e$ has at most as many descendants as $e'$. We do so by induction on the number of leaves/species. \\

\noindent \textbf{Base Case: $k = 4$}\\
\noindent In this case, the two unique rooted trees, shown in Figure \ref{fig:caterpillar_and_balanced_tree}, both have $\alpha_1 = 2$ for their only nontrivial bipartition, so the statement is trivially true. Again we are ignoring trivial bipartitions and possibly one of the edges adjacent to the root.\\

\noindent \textbf{Inductive Step (Caterpillar)}\\
\noindent Suppose we have a tree $T$ with $k > 4$ leaves/species. There are two cases:

\begin{enumerate}[label=(\alph*)]
    \item \textit{One child of root is a leaf}. If $T'$ is the non-leaf subtree of the root, apply the inductive hypothesis to $T'$. Let $\{\alpha_i\}_{i=1}^{k-4}$ be the descendant counts of $T'$. By induction, we can reindex $\{\alpha_i\}_{i=1}^{k-4}$ so that $\alpha_i \leq i+1$. Moreover, as all nontrivial bipartitions have at most $k-2$ leaves in one part, trivially $\alpha_{k-3} \leq k-2$, so we are done. Namely, we can just pair the remaining nontrivial edge of T, with the top nontrivial edge of our caterpillar tree.

    \item \textit{Neither child of root is a leaf}. Let $T', T''$ be the two subtrees of the root. If these have $r,m$ leaves respectively, where $r,m \in \{2,3,..., k-2\}$ and $r+m=k$, let $\{\kappa_i\}_{i=1}^{r-3}$ and $\{\beta_j\}_{j=1}^{m-3}$ be their descendant counts. By induction, these can be ordered so that $\kappa_i \leq i+1$ and $\beta_j \leq j+1$. Clearly then defining

    \[\alpha_i = \begin{cases}
        \kappa_i, & 1 \leq i \leq r-3, \\
        \beta_{i-r+3}, & r-3 < i  \leq k-6. \\
    \end{cases}\]

    \noindent Namely, we just list the edges of $T'$ first, and then the edges of $T''$. This accounts for all but three of the edges of $T$: one of the edges $e$ adjacent to the root of $T$, one edge $e'$ in $T'$, and one edge $e''$ in $T''$.
    \indent Note that the descendant counts $\alpha_{e'}, \alpha_{e''}$ of $e', e''$ are at most $r-1, m-1$ respectively. Moreover
    \[r-1 = k-m-1 \leq k-3, \qquad m-1 = k-r-1 \leq k-3\]
    \noindent as both $T', T''$ have at least two leaves. Equality holds if $m=2$ or $r=2$ respectively, so for $k > 4$ equality cannot hold for both simultaneously. Hence without loss of generality assume $\alpha_{e'} \leq k-4$. So we can let $\alpha_{k-5}, \alpha_{k-4}$ correspond to edges $e', e''$ respectively. Trivially $e$ has descendant count at most $k-2$, as all nontrivial bipartitions do. Thus we can let $\alpha_{k-3}$ correspond to edge $e$. This completes the proof of Lemma \ref{lem:caterpillar_maximize_increasing_sums}.

\begin{figure}[h]
\centering
\begin{tikzpicture}[scale=0.7]

\node[circle, fill=black, inner sep=1.5pt] (root) at (0,0) {};
\node[circle, fill=black, inner sep=1.5pt] (n1) at (-2,-1.2) {};
\node[circle, fill=black, inner sep=1.5pt] (n2) at (2,-1.2) {};
\node[circle, fill=black, inner sep=1.5pt] (n3) at (-3,-2.4) {};
\node[circle, fill=black, inner sep=1.5pt] (n4) at (-1,-2.4) {};
\node[circle, fill=black, inner sep=1.5pt] (n5) at (1,-2.4) {};
\node[circle, fill=black, inner sep=1.5pt] (n6) at (3,-2.4) {};
\node[circle, fill=black, inner sep=1.5pt] (n7) at (-3.5,-3.6) {};
\node[circle, fill=black, inner sep=1.5pt] (n8) at (-2.5,-3.6) {};
\node[circle, fill=black, inner sep=1.5pt] (n9) at (-1.5,-3.6) {};
\node[circle, fill=black, inner sep=1.5pt] (n10) at (-0.5,-3.6) {};
\node[circle, fill=black, inner sep=1.5pt] (n11) at (0.5,-3.6) {};
\node[circle, fill=black, inner sep=1.5pt] (n12) at (1.5,-3.6) {};

\draw (root) -- (n1) node[midway, above left] {\tiny $\alpha_e = 4$};
\draw (root) -- (n2) ;
\draw (n1) -- (n3) node[midway, above left] {\tiny $\alpha_{e'} = 2$};
\draw (n1) -- (n4) node[midway, above right] {\tiny 2};
\draw (n2) -- (n5) node[midway, above left] {\tiny $\alpha_{e''} = 2$};
\draw (n2) -- (n6);
\draw (n3) -- (n7);
\draw (n3) -- (n8);
\draw (n4) -- (n9);
\draw (n4) -- (n10);
\draw (n5) -- (n11);
\draw (n5) -- (n12);
\end{tikzpicture}
\caption{Example of the inductive step in Lemma \ref{lem:caterpillar_maximize_increasing_sums}}
\label{fig:inductive_step_case_b}
\end{figure}
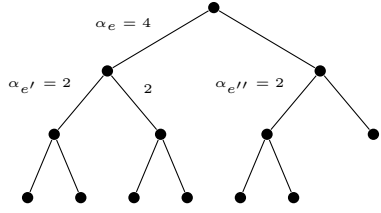
\end{enumerate}
\end{proof}

\noindent We can provide a partial complement to the above Lemma \ref{lem:caterpillar_maximize_increasing_sums} analyzing the other extreme. More concretely, we will show that the more balanced the tree is, the smaller these sums tend to be. We will not use this result in this paper, but it adds to the idea that the balanced tree and the caterpillar tree are in a sense the two extreme points in our space of trees.

\begin{lemma}
 Suppose $\{\alpha_{i}\}_{i=1}^{2k-2}$ are the descendant counts of our rooted binary tree $T$ corresponding to \textit{all} the edges of our tree (not just ones corresponding to nontrivial bipartitions). Moreover, suppose $f: \N \to \R$ is nondecreasing and convex. Then $\sum_i f(\alpha_{i})$ is minimized when $T$ is the balanced tree.
 \label{lem:balanced_trees_minimize_increasing_convex_sums}
\end{lemma}

\begin{proof}[Proof of Lemma \ref{lem:balanced_trees_minimize_increasing_convex_sums}]
\noindent We will describe a natural greedy balancing operation on the tree and prove it decreases $\sum f(\alpha_i)$. Recall that a \textit{cherry} in a binary tree refers to a subtree consisting of two leaves.

Say that a given vertex $v \in V(T)$ is \textit{unbalanced} if the number of leaves in its two subtrees differs by more than one. In order to make the tree more balanced, we will construct a new tree $T'$ by pruning one of the cherries of the larger subtree and attaching it to one of the leaves of the smaller subtree.

To locate which subtree to prune, start at $v$ and traverse the edge leading to the larger subtree. Iteratively descend the tree in this fashion until you reach a leaf $\ell$. At this point, necessarily $\ell$ and its sibling form a cherry. If this were not the case, then the sibling of $\ell$ forms a larger subtree, and we would have traversed into that subtree rather than towards $\ell_p$. This will be the cherry we prune. 

To locate the leaf at which to reattach, start at $v$ and traverse the edge leading to the smaller subtree. Iterate this procedure until you reach a leaf $\ell_a$. Attach the cherry at this leaf $\ell$ to form a new tree $T'$ with the same number of leaves overall as $T$. An illustration of this algorithm is provided in Figure \ref{fig:balancing_algorithm_trees}. 

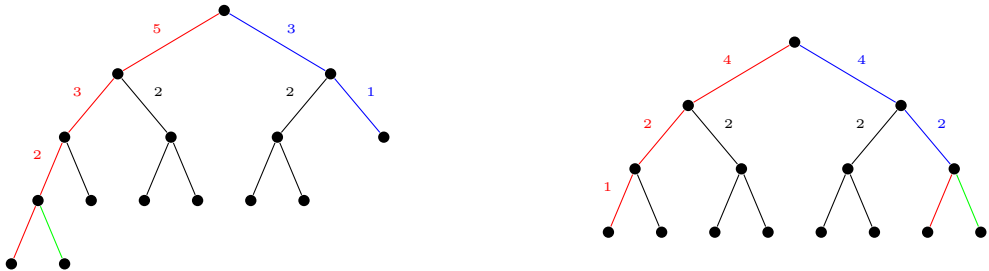
\begin{figure}[h]
\centering
\begin{minipage}{0.48\textwidth}
\centering
\begin{tikzpicture}[scale=0.7]

\node[circle, fill=black, inner sep=1.5pt] (root) at (0,0) {};
\node[circle, fill=black, inner sep=1.5pt] (n1) at (-2,-1.2) {};
\node[circle, fill=black, inner sep=1.5pt] (n2) at (2,-1.2) {};
\node[circle, fill=black, inner sep=1.5pt] (n3) at (-3,-2.4) {};
\node[circle, fill=black, inner sep=1.5pt] (n4) at (-1,-2.4) {};
\node[circle, fill=black, inner sep=1.5pt] (n5) at (1,-2.4) {};
\node[circle, fill=black, inner sep=1.5pt] (n6) at (3,-2.4) {};
\node[circle, fill=black, inner sep=1.5pt] (n7) at (-3.5,-3.6) {};
\node[circle, fill=black, inner sep=1.5pt] (n8) at (-2.5,-3.6) {};
\node[circle, fill=black, inner sep=1.5pt] (n9) at (-1.5,-3.6) {};
\node[circle, fill=black, inner sep=1.5pt] (n10) at (-0.5,-3.6) {};
\node[circle, fill=black, inner sep=1.5pt] (n11) at (0.5,-3.6) {};
\node[circle, fill=black, inner sep=1.5pt] (n12) at (1.5,-3.6) {};

\node[circle, fill=black, inner sep=1.5pt] (n13) at (-4,-4.8) {};
\node[circle, fill=black, inner sep=1.5pt] (n14) at (-3,-4.8) {};

\draw[red] (root) -- (n1) node[midway, above left] {\tiny $5$};
\draw[blue] (root) -- (n2) node[midway, above right] {\tiny $3$};
\draw[red] (n1) -- (n3) node[midway, above left] {\tiny $3$};
\draw (n1) -- (n4) node[midway, above right] {\tiny 2};
\draw (n2) -- (n5) node[midway, above left] {\tiny $2$};
\draw[blue] (n2) -- (n6) node[midway, above right] {\tiny 1};
\draw[red] (n3) -- (n7) node[midway, above left] {\tiny $2$};
\draw (n3) -- (n8);
\draw (n4) -- (n9);
\draw (n4) -- (n10);
\draw (n5) -- (n11);
\draw (n5) -- (n12);
\draw[red] (n7) -- (n13);
\draw[green] (n7) -- (n14);

\end{tikzpicture}
\end{minipage}
\hfill
\begin{minipage}{0.48\textwidth}
\centering
\begin{tikzpicture}[scale=0.7]

\node[circle, fill=black, inner sep=1.5pt] (root) at (0,0) {};
\node[circle, fill=black, inner sep=1.5pt] (n1) at (-2,-1.2) {};
\node[circle, fill=black, inner sep=1.5pt] (n2) at (2,-1.2) {};
\node[circle, fill=black, inner sep=1.5pt] (n3) at (-3,-2.4) {};
\node[circle, fill=black, inner sep=1.5pt] (n4) at (-1,-2.4) {};
\node[circle, fill=black, inner sep=1.5pt] (n5) at (1,-2.4) {};
\node[circle, fill=black, inner sep=1.5pt] (n6) at (3,-2.4) {};
\node[circle, fill=black, inner sep=1.5pt] (n7) at (-3.5,-3.6) {};
\node[circle, fill=black, inner sep=1.5pt] (n8) at (-2.5,-3.6) {};
\node[circle, fill=black, inner sep=1.5pt] (n9) at (-1.5,-3.6) {};
\node[circle, fill=black, inner sep=1.5pt] (n10) at (-0.5,-3.6) {};
\node[circle, fill=black, inner sep=1.5pt] (n11) at (0.5,-3.6) {};
\node[circle, fill=black, inner sep=1.5pt] (n12) at (1.5,-3.6) {};

\node[circle, fill=black, inner sep=1.5pt] (n13) at (3.5,-3.6) {};
\node[circle, fill=black, inner sep=1.5pt] (n14) at (2.5,-3.6) {};

\draw[red] (root) -- (n1) node[midway, above left] {\tiny $4$};
\draw[blue] (root) -- (n2) node[midway, above right] {\tiny $4$};
\draw[red] (n1) -- (n3) node[midway, above left] {\tiny $2$};
\draw (n1) -- (n4) node[midway, above right] {\tiny 2};
\draw (n2) -- (n5)node[midway, above left] {\tiny 2};
\draw[blue] (n2) -- (n6) node[midway, above right] {\tiny $2$};
\draw[red] (n3) -- (n7) node[midway, above left] {\tiny $1$};
\draw (n3) -- (n8);
\draw (n4) -- (n9);
\draw (n4) -- (n10);
\draw (n5) -- (n11);
\draw (n5) -- (n12);
\draw[green] (n6) -- (n13);
\draw[red] (n6) -- (n14);

\end{tikzpicture}
\end{minipage}

\caption{Side-by-side illustration of the initial tree $T$ (left) and the tree $T'$ after balancing (right). The pruning path $\{e_i^p\}$ in red and the attachment path $\{e_i^a\}$ in blue are highlighted}
\label{fig:balancing_algorithm_trees}
\end{figure}

Suppose that $\{\alpha_i\}_{i=1}^{2k-2}, \{\alpha_i'\}_{i=1}^{2k-2}$ are the descendant counts of $T, T'$ respectively. We claim that $\sum f(\alpha_i') \leq \sum f(\alpha_i)$, namely that this balancing operation can only decrease the overall sum. To see why, let
\[v = u_0 \xrightarrow{e^p_1} ... \xrightarrow{e^p_n} u_n = \ell_p, \quad v = w_0 \xrightarrow{e^a_1} ... \xrightarrow{e^a_m} w_m = \ell_a\]
\noindent be the paths we traverse the tree in our pruning (from $v \to \ell_p$) and grafting (from $v \to \ell_a$) operations respectively. Note that $\alpha_e = \alpha'_e$ for all edges $e \notin \{e_i^p\} \cup \{e_i^a\}$ not along either of these paths. Moreover, every edge along the pruning path loses exactly one descendant leaf and every edge along the attachment path gains exactly one leaf.  The edges in $T,T'$ are exactly the same, except for the relocation of the cherry, and hence we pair the descendant counts $\alpha_e, \alpha'_e$

\[\sum_i (f(\alpha_i) - f(\alpha'_i))=  \sum_{i=1}^{n-1} [f(\alpha_{e_i^p}) -  f(\alpha_{e_i^p} - 1)] + \sum_{i=1}^{m-1} [f(\alpha_{e_i^a}) -  f(\alpha_{e_i^a} + 1)].\]

\noindent Moreover,  by construction $m \leq n$ and $\alpha_{e_i^p} > \alpha_{e_i^a}$. This is because we always choose the larger subtree for pruning and the smaller for attaching. Thus $\alpha_{e_1^p} > \alpha_{e_1^a}$ and hence by induction
\[\alpha_{e_i}^p \geq \bigg\lceil \frac{\alpha_{e_{i-1}^p}}{2} \bigg \rceil > \bigg\lfloor \frac{\alpha_{e_{i-1}^a}}{2} \bigg \rfloor \geq \alpha_{e_i^a}\]
\noindent so $\alpha_{e_i^p} > \alpha_{e_i^a}$ is true for all $i$. Using this, we can see

\[= \sum_{i=1}^{m-1} [f(\alpha_{e_i^p}) -  f(\alpha_{e_i^p} - 1)] + [f(\alpha_{e_i^a}) -  f(\alpha_{e_i^a} + 1)] + \sum_{i=m}^{n-1} [f(\alpha_{e_i^p}) -  f(\alpha_{e_i^p} - 1)].\]

\noindent Each term in the first summand is nonnegative as $f$ is convex, and each term in the second summand is nonnegative as $f$ is increasing. Hence $\sum (f(\alpha_i) - f(\alpha_i')) \geq 0$ and the result follows.

Lastly, we argue that any tree can be converted into a balanced tree by iterating this procedure. If the subtrees of a given unbalanced vertex have size $L,R$ respectively where $L-R \geq 2$ then clearly after this procedure they have sizes $L-1, R+1$ respectively with imbalance $(L-1) - (R+1) = (L-R) - 2$. Hence the imbalance decreases by two, and if we iterate this procedure enough times we can get the imbalance to be $(L-R) \; mod(2)$, namely the vertex is balanced. Moreover, the descendant counts of any vertices not contained in one of the subtrees of $v$ are left unchanged. Hence we can start at the root and iteratively apply the procedure until the root is balanced. Then we can work our way down the tree. 
\end{proof}

This type of iterative subtree balancing is reminiscent of measures of tree balance in phylogenetics, where metrics such as the Colless or Sackin indices quantify how balanced a tree is relative to the maximally balanced or maximally imbalanced (caterpillar) configurations. Tree balance statistics have long been used in phylogenetics to characterize the shapes of evolutionary trees, compare them to null models, and infer evolutionary processes such as speciation or extinction biases \cite{kersting2025tree}. 

Classical tree-balancing procedures in computer science—such as AVL trees \cite{adelson1962algorithm} and weight-balanced trees \cite{nievergelt1973binary}—use local rotations to enforce constraints on subtree depths or sizes. Our algorithm shares the same conceptual goal of redistributing leaves to reduce imbalance, but it is purely combinatorial and does not maintain any search property.

\newpage 
\subsection{Proof Lemma \ref{lem:deterministic_balancing_lemma}: Deterministic Balancing Lemma}
\label{app:deterministic-balancing-proof}

Recall we aim to prove:\\

\noindent \textbf{Lemma:} Let $k \geq 4$ be fixed. Then for any fixed $T$
    \[Z^T_{i} + Z^T_{k-i} \leq_{st} Z^T_{j} + Z^T_{k-j},  \quad  1 \leq i \leq j \leq \frac{k}{2}.\]
    \noindent Namely, the more balanced the subtrees are below edge $e$, the more lineages typically enter $e$.

\begin{proof}[Proof Lemma \ref{lem:deterministic_balancing_lemma}]

 The statement is vacuously true for $k=2,3$, so for induction assume it is true for all $k' < k$. Let $S_{k,m}^T := Z_m^T + Z^T_{k-m}$ denote the total number of lineages entering an edge $e$ with subtrees of size $m,k-m$, the setup in Figure \ref{fig:coalescent_tree}. Let $\tau_{k,m}$ be the time of the first coalescent event in either of our two subpopulations of size $m,k-m$, allowing the possibility $\tau_{k,m} > T$.

  The basic idea is that if we start with $k$ lineages, we can reduce to $k-1$ lineages by just waiting for the first coalescence event to occur in either subpopulation. There is of course the chance no coalescent event occurs, so we must handle this separately. The proof then relies on a few main ideas. Firstly, for more balanced splits we tend to wait longer until the first coalescent event.
  \begin{proofclaim}
      For fixed $k$, $\tau_{k,m}$ is stochastically increasing in $m$ for $1 \leq m \leq k/2$.
      \label{claim:even_split_waits_longest_for_first_coalescent}
  \end{proofclaim}
  \noindent Next, conditional on a specific first coalescent time, the more balanced our split is, the more lineages tend to remain uncoalesced by the end of the branch.
  \begin{proofclaim}
      For fixed $k, x,t$, $\P(S_{k,m}^T \geq x \; | \; \tau_{k,m} = t)$ is an increasing function of $m$ for $1 \leq m \leq k/2$. 
\label{claim:cond_first_coalescent_even_split_dominates}
  \end{proofclaim}
  \noindent Lastly, the more time we wait until the first coalescent event, the less time the remaining lineages have to coalesce and the more lineages we tend to have by the end of the branch. 
  %
  \begin{proofclaim}
      For fixed $k, x, m$, $\P(S_{k,m}^T \geq x \; | \; \tau_{k,m} = t)$ is an increasing function of $t$.
      \label{claim:more_time_to_coalesce_is_better}
  \end{proofclaim}
  %
  \begin{proof}[Proof of Claim \ref{claim:even_split_waits_longest_for_first_coalescent}]
        Let $h_{k,m} := {m \choose 2} + {k-m \choose 2}$. Then $\tau_{k,m} \sim \mathrm{Exp}\left(h_{k,m}\right)$. This comes from the fact that coalescents occur at rate ${m \choose 2}$ in the population of size $m$ and at rate ${k-m \choose 2}$ in the population of size $k-m$. Together, these produce coalescent events at rate $h_{k,m}$. Since $\tau_{k,m} \sim \exp(h_{k,m})$
        \[\P(\tau_{k,m} > t) = \exp\left(-th_{k,m}\right)\]
        \noindent As $h_{k,m}$ is decreasing in $m$ for $1 \leq m \leq k/2$, the above is increasing in $m$. Hence $\tau_{k,m}$ is stochastically increasing  in $m$. This proves Claim 1.
  \end{proof}

  \begin{proof}[Proof of Claim \ref{claim:cond_first_coalescent_even_split_dominates}]
      For $t > T$ then we $\P(S_{k,m}^T \geq x | \tau_{k,m} = t) = 1$ so the claim is trivially true. For $t < T$, by the strong memoryless property of the exponential distribution underlying Kingman's coalescent, we know
        \begin{equation}
            \P(S_{k,m}^T \geq x \; | \; \tau_{k,m} = t) = p_{k,m} \cdot \P(S_{k-1,m-1}^{T-t} \geq x) + (1-p_{k, m}) \cdot \P(S_{k-1,m}^{T-t} \geq x)
            \label{eqn:split_on_first_coalescent_event}
        \end{equation}

        \noindent where
        \[p_{k,m} = \P\left[\exp\left({m \choose 2}\right) \leq \exp\left({k-m \choose 2}\right)\right] = \frac{{m \choose 2}}{{m \choose 2} + {k-m \choose 2}}\]
        \noindent is the probability the first coalescence occurs in the population of size $m$. This essentially just says that after the first coalescent event occurs in either subpopulation, what remains is still a Kingman's coalescent just with one less lineage. By our inductive hypothesis, for $1 \leq m \leq k/2 - 1$
        \begin{align*}
             & \P(S_{k,m}^T \geq x \; | \; \tau_{k,m} = t)\\
             &  = p_{k,m} \cdot \P(S_{k-1,m-1}^{T-t} \geq x) + (1-p_{k, m}) \cdot \P(S_{k-1,m}^{T-t} \geq x) \\
            & \leq \P(S_{k-1,m}^{T-t} \geq x) \\ 
            & \leq  p_{k,m+1} \cdot \P(S_{k-1,m}^{T-t} \geq x) + (1-p_{k, m+1}) \cdot \P(S_{k-1,m+1}^{T-t} \geq x)\\
            & = \P(S_{k,m+1}^T \geq x \; | \; \tau_{k,m+1} = t)
        \end{align*}
        \noindent so that $\P(S_{k,m}^T \geq x \; | \; \tau_{k,m} = t)$ is increasing in $m$ for each fixed $t$. This proves Claim 2.
  \end{proof}

  \begin{proof}[Proof of Claim \ref{claim:more_time_to_coalesce_is_better}]
      By equation \eqref{eqn:split_on_first_coalescent_event} it suffices to show for fixed $k,m$ that $S^T_{k,m}$ is stochastically decreasing in $T$. But this follows from a simple coupling argument. For $s \leq T$, couple $S_{k,m}^T$ to $S_{k,m}^s$ by using the same exponential clocks for the coalescent events. At time $s$ the number of lineages remaining agree, and $S^T_{k,m}$ can only decrease from here if more coalescents occur between times $s$ and $T$. This proves Claim 3. 
  \end{proof}

  \noindent With the claims proved, we are ready to finally finish our proof of Lemma \ref{lem:deterministic_balancing_lemma}. Note that our above claims are exactly the conditions of Lemma \ref{lem:mixtures_and_stochastic_dominance} where $\{S_{k,m}^T\}_{m=1}^{\lfloor k/2\rfloor}$ are playing the role of our $V_i$ and $\{\tau_{k,m}\}_{m=1}^{\lfloor k/2\rfloor}$ are playing the role of our $\Theta_i$. Hence the result just follows from applying this lemma. This completes the proof of Lemma \ref{lem:deterministic_balancing_lemma}.

\end{proof}

\subsection{Proof of Lemma \ref{lem:balanced_is_worse_case}: Balanced is Worst-Case}
\label{app:balanced-worst-case-proof}

Recall we aim to prove:\\

\noindent \textbf{Lemma:} If $\cB_k$ is the balanced tree with $k$ leaves and all branches length $T_{\mathrm{min}}$, then $X_{\cB_k} \geq_{st} X_\cT$ for any other tree $\cT$ with $k$ leaves and minimum branch length $T_{\mathrm{min}}$.\\

\noindent To simplify notation a bit, denote $X_m := X_{\cB_m}$.  To start, we will prove two extensions of our deterministic balancing lemma \ref{lem:deterministic_balancing_lemma}. The first allows the difference between the subpopulation sizes to be random. Heuristically, it  says that conditional on the total number of species in the two subtrees below a given vertex, the more even the split tends to be, the more lineages tend to reach the vertex.
\begin{lemma}
    Suppose for $s \in \N$ that $0 \leq_{st} D \leq_{st} D' < s$ and $D,D' \equiv s \; \mathrm{mod}(2)$. Then
    \[Z_{\frac{s + D}{2}} + Z_{\frac{s - D}{2}} \geq_{st} Z_{\frac{s + D'}{2}} + Z_{\frac{s - D'}{2}}.\]  
\label{lem:stochastic_dominance_with_differences}
\end{lemma}
\begin{proof}[Proof of Lemma \ref{lem:stochastic_dominance_with_differences}]
    Since the sum $s = (s+d)/2 + (s-d)/2$ is fixed, our deterministic balancing lemma \ref{lem:deterministic_balancing_lemma} implies that
    \[Y_d := Z_{\frac{s+d}{2}} + Z_{\frac{s-d}{2}}\]
    \noindent is decreasing in $d$ with respect to $\leq_{st}$ for $0 \leq d < s$. Lemma \ref{lem:mixtures_and_stochastic_dominance} then implies the result. 
\end{proof}

\noindent The next extension generalizes this even further by allowing the sum of the subpopulation sizes to be random as well. Heuristically, it captures the same idea as before: more balanced subpopulations tend to produce less coalescence.

  \begin{lemma}[Balancing Lemma] 
    If $A,A', B, B'$ are mutually independent, positive integer-valued random variables so that $A \leq_{lr} A', B \leq_{lr} B'$, then
    \[Z_{A + B} + Z_{A' + B'} \leq_{st} Z_{A + B'} + Z_{A' + B}. \]
    \label{lem:balancing_lemma}
    \end{lemma}
    \noindent Note when $A,B,A', B'$ are all deterministic, this reduces to Lemma \ref{lem:deterministic_balancing_lemma}. 

    \begin{proof}[Proof of Balancing Lemma \ref{lem:balancing_lemma}]
        For convenience, define
        \[\Delta_A := A' - A, \quad \Delta_B := B' - B, \quad S_A := A' + A, \quad S_B = B' + B.\]
        \noindent Now, note in both $Z_{A + B} + Z_{A' + B'}$ and $Z_{A + B'} + Z_{A' + B}$ the total number of lineages $S = S_A + S_B$ is the same. Hence by conditioning on $S_A,S_B, |\Delta_A|,$ and $|\Delta_B|$, Lemma \ref{lem:stochastic_dominance_with_differences} implies it suffices to show $D' \geq_{st} D$ where $D,D'$ are the conditional differences
        \[D' = |\Delta_A + \Delta_B| \mid \{|\Delta_A|, |\Delta_B| , S_A, S_B \}, \qquad D = |\Delta_A - \Delta_B| \mid \{|\Delta_A|, |\Delta_B| , S_A, S_B \}.\]
        \noindent After this conditioning, $\{D, D'\} = \{|\Delta_A| + |\Delta_B|, ||\Delta_A| - |\Delta_B||\}$. Let
        
        \[q = \P(D' = |\Delta_A| + |\Delta_B|) = \P(\Delta_A\Delta_B \geq 0 \mid |\Delta_A|, |\Delta_B| , S_A, S_B)\]
        
        \noindent be the probability $D'$  takes on the larger of these two values. Then it suffices to show $q \geq 1/2$. Since $\Delta_A, \Delta_B$ are conditionally independent given $S_A, S_B$
        \[q = p_Ap_B + (1-p_A)(1-p_B) \geq \frac{1}{2} \iff \left(p_A - \frac{1}{2} \right) \left(p_B - \frac{1}{2} \right) \geq 0,\]
        \noindent where $p_A = \P(\Delta_A \geq 0 \mid |\Delta_A|, S_A)$ and likewise for $p_B$. We will show that $p_A, p_B \geq 1/2$, completing the proof. Note conditional on $\{|\Delta_A| = m, S_A = s\}$ then $p_A \geq 1/2$ occurs precisely when
        \[\P\left(A'=\frac{s+m}{2}, A = \frac{s-m}{2}\right) \geq \P\left(A'=\frac{s-m}{2}, A = \frac{s+m}{2}\right).\]
        \noindent Since $A,A'$ are independent, this follows from $A \leq_{lr} A'.$ Similarly, $p_B \geq 1/2$. Thus $D' \geq_{st} D$. 
    \end{proof}


\noindent Now we are ready to return to the main result. First, we show that the result follows from the following stochastic extension of Lemma \ref{lem:deterministic_balancing_lemma}:

\begin{lemma}
    For fixed $k \in \N$, then $Z_{X_i} + Z_{X_{k-i}} \leq_{st} Z_{X_{\lfloor k/2\rfloor}} + Z_{X_{\lceil k/2\rceil}}$ for $1 \leq i \leq \lfloor k/2\rfloor$. 
    \label{lem:the_more_balanced_the_better}
\end{lemma}

\noindent We conjecture that $Z_{X_i} + Z_{X_{k-i}}$ is increasing in $i$ rather than just maximized at the even split. But for our current purposes, the above suffices. Assuming Lemma \ref{lem:the_more_balanced_the_better} for now, it is easy to prove our desired Lemma \ref{lem:balanced_is_worse_case}.

\begin{proof}[Reduction of Lemma \ref{lem:balanced_is_worse_case} to Lemma \ref{lem:the_more_balanced_the_better}]
 We induct on the number of species $k$.\\

\noindent \textbf{Base Case ($k = 4$)}\\
\noindent Again there are only two different tree topologies, the ones in Figure \ref{fig:caterpillar_and_balanced_tree}. Lemma \ref{lem:deterministic_balancing_lemma} shows the even split $Z_2 + Z_2' \geq_{st} Z_1  + Z_3'$, and clearly $Z_1 + Z_3' \geq_{st} Z_1 + Z_{X_3}'$ by Corollary \ref{corollary:stochastic_dominance_composition}.\\

\noindent \textbf{Inductive Step}\\
\noindent Suppose $\cT$ is a tree with $k > 4$ leaves. As usual, the main idea is to reduce to smaller trees by looking at the subtrees of the root. We perform the induction in two steps: we first reduce to the case the two subtrees of the root are themselves balanced trees, and then we show in such a setting the worst-case is when the two subtrees have the same number of leaves. 

Let $L$ and $R$ be the two subtrees of the root of $\cT$ which are of sizes $m$ and $k-m$ respectively. By our inductive hypothesis, $X_L \leq_{st} X_{\cB_m}$ and $X_R \leq_{st} X_{\cB_{k-m}}$. Let $\cT'$ be the tree whose two subtrees are $\cB_m, \cB_{k-m}$, and whose branch lengths are all $T_{\mathrm{min}}$. Then as $X_\cT = Z_{X_L} + Z_{X_R}$ we see Corollary \ref{corollary:stochastic_dominance_composition} implies $X_{\cT'} \geq_{st} X_{\cT}$ stochastically. Hence all that is left to show is that $X_{\cB_k} \geq_{st} X_{\cT'}$. But this is exactly just the content of Lemma \ref{lem:the_more_balanced_the_better}. Hence Lemma \ref{lem:balanced_is_worse_case} follows from Lemma \ref{lem:the_more_balanced_the_better}.
\end{proof}

\noindent Now, all that is left to do is prove Lemma \ref{lem:the_more_balanced_the_better}.

\begin{proof}[Proof of Lemma \ref{lem:the_more_balanced_the_better}]

 We prove this statement by induction on $k$. Again the base case is trivial
    \[Z_1 + Z_{X_3} \leq_{st} Z_1 + Z_3 \leq_{st} Z_2 + Z_2 = Z_{X_2} + Z_{X_2}.\]
    \noindent Now suppose the statement is true for $k' < k$. Note
    
    \[X_m = Z_{X_{\lfloor m/2 \rfloor}} + Z_{X_{\lceil m/2 \rceil}}, \qquad X_{k-m} = Z_{X_{\lfloor (k-m)/2 \rfloor}} + Z_{X_{\lceil (k-m)/2 \rceil}}.\]

    \noindent By Lemma \ref{lem:xi_zxi_increasing_in_lro} we have for $1 \leq m \leq k/2$

    \[Z_{X_{\lfloor m/2 \rfloor}} \leq_{lr} Z_{X_{\lceil m/2 \rceil}} \leq_{lr} Z_{X_{\lfloor (k-m)/2 \rfloor}} \leq_{lr} Z_{X_{\lceil (k-m)/2 \rceil}}. \]
    
    \noindent Note that if $k$ is even then
    \[\lfloor m/2 \rfloor + \lceil (k-m)/2 \rceil = \lceil m/2 \rceil + \lfloor (k-m)/2 \rfloor = \frac{k}{2}.\]
    \noindent The balancing lemma \ref{lem:balancing_lemma} then implies
    \[Z_{X_m} + Z_{X_{k-m}} \leq_{st} \left(Z_{Z_{X_{\lfloor m/2 \rfloor}} + Z_{X_{\lceil (k-m)/2 \rceil}}}\right) + \left(Z_{Z_{X_{\lceil m/2 \rceil}} + Z_{X_{\lfloor (k-m)/2 \rfloor}}}\right).\]
    \noindent The inductive hypothesis applied to $k' = k/2$ shows the subscripts of each term are maximized at $m = k/2$. Hence Corollary \ref{corollary:stochastic_dominance_composition} implies
    \[ \leq_{st} \left(Z_{Z_{X_{\lfloor k/4 \rfloor}} + Z_{X_{\lceil k/4 \rceil}}}\right) + \left(Z_{Z_{X_{\lceil k/4 \rceil}} + Z_{X_{\lfloor k/4 \rfloor}}}\right) = Z_{X_{k/2}} + Z_{X_{k/2}},\]
    \noindent as desired. Likewise, if $k$ is odd then
    \[\lfloor m/2 \rfloor + \lfloor (k-m)/2 \rfloor = \frac{k-1}{2}, \quad and \quad \lceil m/2 \rceil + \lceil (k-m)/2 \rceil = \frac{k+1}{2}\]
    \noindent so we can rearrange using the balancing lemma \ref{lem:balancing_lemma} as
    \[Z_{X_m} + Z_{X_{k-m}} \leq_{st} \left(Z_{Z_{X_{\lfloor m/2 \rfloor}} + Z_{X_{\lfloor (k-m)/2 \rfloor}}}\right) + \left(Z_{Z_{X_{\lceil m/2 \rceil}} + Z_{X_{\lceil (k-m)/2 \rceil}}}\right)\]
    \noindent Applying induction with $k' = (k-1)/2$ on the left term and $k' = (k+1)/2$ on the right term we see both terms are maximized at $m = \lfloor k /2 \rfloor = (k-1)/2$. Hence
    \[\leq_{st} \left(Z_{Z_{X_{\lfloor (k-1)/4 \rfloor}} + Z_{X_{\lfloor (k+1)/4 \rfloor}}}\right) + \left(Z_{Z_{X_{\lceil (k-1)/4 \rceil}} + Z_{X_{\lceil (k+1)/4 \rceil}}}\right) = Z_{X_{(k-1)/2}} + Z_{X_{(k+1) / 2}}\]
    \noindent as desired. 
    \end{proof}

\section{Asymptotic Results}
\label{app:asymptotic-results}

\subsection{Proof of Lemma \ref{lem:asymptotics_g_k1(T)}: Asymptotics $g_{k,1}(T)$}
\label{app:coalescent-asymptotics-proof}

In this section, we develop some asymptotics for the time $g_{k,1}(T)$ to coalescence of Kingman's coalescent, specialized to the topic of interest. More properties of this coalescent time are discussed in \cite{MoehlePitters2015AbsorptionTimeTreeLength} and \cite{ChenChen2013Asymptotic}. The paper \cite{kersting2017timeabsorptionlambdacoalescents} further generalizes some of these ideas to $\Lambda$-coalescents.

 We will work in slightly more generality than this paper requires. Namely, suppose that we have positive \textbf{distinct} rates $\{\lambda_i\}_{i=1}^\infty$. Define the random variables $\{S_k\}_{k =1}^\infty$ by $S_k = \sum_{i=1}^k X_i$ for $X_i \sim \mathrm{Exp}(\lambda_i)$ independent (namely $X_i$ has density $\tau_i(x) := \lambda_i \exp(-\lambda_i x)$). Moreover, define $F_k(T) := \P(S_k \leq T)$ to be their CDFs and $f_k(T)$ their densities. If we define
\[C_{i,k} := \prod_{\substack{m=1 \\ m \neq i}}^k \frac{\lambda_m}{\lambda_m - \lambda_i}\]
\noindent then we have the following well-known explicit formulas for the CDFs and densities
\begin{equation}
    F_k(T) = 1-\sum_{i=1}^k C_{i,k} \exp(-\lambda_i T), \qquad f_k(T) = \sum_{i=1}^k C_{i,k} \lambda_i \exp(-\lambda_i T).
    \label{eqn:cdf_density_hypoexponential}
\end{equation}
\noindent See \cite{Yanev_2020}, for example. Moreover, we have the natural recursive relation
\begin{equation}
    F_k(T) = (F_{k-1} \ast \tau_k)(T) := \int_0^T F_{k-1}(x) \lambda_k \exp(-\lambda_k (T-x)) \; dx.
    \label{eqn:recursion_convolution}
\end{equation}
Since $S_k$ are monotonically increasing, we see $S_k \uparrow S_\infty := \sum_{i=1}^\infty X_i$ almost surely, where the limit may be infinite. To quantify the rate of convergence, define the partial sums $s_k  = \sum_{i=1}^k \lambda_i$, the remainder $R_k := \sum_{i=k+1}^\infty X_i$, and the remainder's expectation $r_k := \E R_k= \sum_{i=k+1}^\infty \lambda_i^{-1}$. We start by getting some explicit formulas for the derivatives of the CDF in terms of these rates $\{\lambda_i\}$:

\begin{lemma}
    If we have positive, distinct rates $\{\lambda_i\}_{i = 1}^\infty$:

    \begin{enumerate}[label=(\roman*)]
        \item At zero, the one-sided derivatives $F_k^{(m)}(0) = 0$ for $0 \leq m \leq k-1$. The nonzero derivatives take the form
        \[F_{k}^{{(k + r)}}(0) = (-1)^r \left(\prod_{i=1}^k \lambda_i\right) \left(h_r(\lambda_1,..., \lambda_k)\right),\]
        \noindent where $h_r$ are the complete homogeneous symmetric polynomials
        \[ h_r(\lambda_1, ..., \lambda_k) := \sum_{1\leq i_1\leq i_2 \leq ...\leq i_r\leq k} \prod_{j=1}^r \lambda_{i_j}.\]
        \noindent In particular, for $r=0, 1$ we have
            \[F_k^{(k)}(0) = \prod_{i=1}^k \lambda_i, \qquad F_k^{(k+1)}(0) = -\left(\sum_{i=1}^k \lambda_i\right)\left(\prod_{i=1}^k \lambda_i\right).\]
        \item The derivatives are uniformly bounded. Namely, for fixed $m$
        
        \[||F_k^{(m)}||_\infty \leq M_m = \max_{i=1, ...,m} |F_i^{(m)}(0)|.\]

    \end{enumerate}

    \label{lem:hypoexponential_derivative_properties}
\end{lemma}

\begin{proof}

\noindent (i) It is easy to see that for $m \geq 1$
\[F_k^{(m)}(0) = (-1)^{m-1}\sum_{i=1}^k C_{i,k} \lambda_i^m\]
which vanishes by Lemma 2 of \cite{Yanev_2020} when $m \leq k-1$. To calculate the first nonzero derivatives, we will rely on the recurrence in Equation \ref{eqn:recursion_convolution}. Note that given two smooth functions $g,h : [0, \infty) \to \R$ and their convolution
\[g \ast h(T) := \int_0^T g(T-t) h(t) \; dt = \int_0^T g(t) h(T-t) \; dt.\]
\noindent Hence, by Leibniz's integral rule we have
\begin{equation}
    (g \ast h)'(T) = (g' \ast h)(T) + g(0)h(T) = (g \ast h')(T) + g(T)h(0).
    \label{eqn:derivatives_of_convolutions}
\end{equation}
\noindent By iteratively applying this to $F_k = F_{k-1} \ast \tau_k$, we see
\begin{equation}
F_k^{(m)}(T) = (F_{k-1}^{(m)} \ast \tau_k)(T) + \sum_{j=0}^{m-k} F_{k-1}^{(k-1+j)}(0) \tau^{(m-k-j)}_k(T),\\
\label{eqn:convolution_formula_derivatives}
\end{equation}
\noindent since by $(ii)$ the first $k-2$ derivatives of $F_{k-1}$ vanish. In particular, if $k > m$ we see $F_k^{(m)} = F^{(m)}_{k-1} \ast \tau_k$. Using this, we can prove $F_{k}^{(k+m)}(0)$ has the desired form by induction. Note $F_1(T) = 1-\exp(-\lambda_1 T)$ so $F_1^{(m+1)}(0) = (-1)^{m}\lambda_1 \cdot \lambda_1^m$ as desired. Moreover, $\tau_k^{(r)}(0) = (-1)^r \lambda_k^{r+1}$. Thus by induction and equation \eqref{eqn:convolution_formula_derivatives}

\[F_{k}^{(k+m)}(0) = \sum_{j=0}^m F_{k-1}^{(k-1 + j)}(0)\,\tau^{(m-j)}_{k}(0) = (-1)^m \left(\prod_{i=1}^k \lambda_i\right) \sum_{j=0}^m \lambda_{k}^{m-j} a_{k-1, j} \]
\noindent where $a_{k, m} := h_m(\lambda_1,..., \lambda_k)$ are our symmetric polynomials. Hence it suffices to show these polynomials satisfy the recurrence
\[a_{k,m} = \sum_{j=0}^m \lambda_{k}^{m-j} a_{k-1, j}.\]
\noindent But this is almost immediate, just by grouping terms according to the number of $i_\ell$ that equal $k$
\[
a_{k,m}
= \sum_{1\leq i_1\leq i_2\leq\cdots\leq i_m\leq k}
    \prod_{\ell=1}^m \lambda_{i_\ell}
= \sum_{j=0}^m \lambda_k^{m-j}
    \sum_{1\leq i_1\leq i_2\leq\cdots\leq i_j\leq k-1}
    \prod_{\ell=1}^j \lambda_{i_\ell}
= \sum_{j=0}^m \lambda_k^{m-j} a_{k-1,j},
\]
where the inner sum and product for $j=0$ are understood to equal one.
\noindent Hence
    \[F_{k}^{(k+m)}(0) = (-1)^m a_{k,m} \prod_{i=1}^k \lambda_i\]
\noindent takes the desired form. \\

\noindent (ii) Since we saw in the previous section that for $k > m$ we have $F_k^{(m)} = F_{k-1}^{(m)} \ast \tau_k$, by iteratively applying Young's convolution inequality we get
\[ ||F_{k}^{(m)}||_\infty \leq ||F_{k-1}^{(m)}||_\infty \cdot ||\tau_k||_1 = ||F_{k-1}^{(m)}||_\infty  \Longrightarrow  ||F_{k}^{(m)}||_\infty \leq ||F_m^{(m)}||_\infty, \quad \forall k \geq m.\]
\noindent Hence we can see $||F_k^{(m)}||_\infty \leq \max_{i=1,...,m} ||F_i^{(m)}||_\infty =: M_m < \infty$. 

To quantify $M_m$, note that using the expression for the derivatives of $\tau_k$, we can see equation \eqref{eqn:convolution_formula_derivatives} can be re-expressed as
\[F_k^{(m)}(T) = F_{k-1}^{(m)} \ast \tau_k(T) + F_{k}^{(m)}(0) \exp(-\lambda_k T).\]
\noindent Then Young's convolution inequality implies
\[|F_k^{(m)}(T)| \leq ||F_{k-1}^{(m)}||_\infty \cdot (1-\exp(-\lambda_k T)) + |F_{k}^{(m)}(0)| \exp(-\lambda_k T) \leq \max \{||F_{k-1}^{(m)}||_\infty, |F_{k}^{(m)}(0)|\}.\]
for any $T \geq 0$. Taking the supremum over $T$ yields
    \[||F_k^{(m)}||_\infty \leq \max \{||F_{k-1}^{(m)}||_\infty, |F_{k}^{(m)}(0)|\}.\]
\noindent By induction we see $M_m = \max_{i=1,..., m}|F_i^{(m)}(0)|$ as desired. \\
\end{proof}

\subsubsection{CDF Asymptotics}
\label{app:cdf-asymptotics}
\noindent Now that we have these bounds, the main results follow. We in fact show something slightly stronger: that the relative error in the Taylor series approximation is negligible in the $T = o((s_k/k)^{-1})$ regime no matter where we truncate it. 

\begin{theorem}
    If $m = \arg \min_{i=1,...,k} \lambda_i$, then the CDFs have asymptotics
        \[\begin{cases}
            1-F_k(T)  \sim   C_{m,k}\exp(-\lambda_m T), & T \uparrow \infty \text{ and k fixed,}\\
            F_k(T)  \sim \frac{\prod_{i=1}^k \lambda_i}{k!} T^{k}, &  T = o((s_{k}/k)^{-1}),\\
        \end{cases}\]
        \noindent where $T = o((s_k/k)^{-1})$ denotes any sequence of pairs $(T_n, k_n)$ so that $\lim_{n \to \infty} s_{k_n}T_n/k_n = 0$. More generally, in the same $T = o((s_k/k)^{-1})$ regime and for any finite $m \geq 0$
        \[\frac{F_k(T) - \sum_{r=0}^{m} a_{r,k}}{a_{m, k}} \to 0, \qquad a_{r,k} := (-1)^r \frac{h_{r}(\lambda_1,..., \lambda_k)\prod_{i=1}^k \lambda_i }{(k+r)!} T^{k+r}.\]
\end{theorem}

\begin{proof}
 For $k$ fixed, the exponential in equation \ref{eqn:cdf_density_hypoexponential} with the smallest rate decays the slowest, so $1-F_k(T) \sim C_{m,k}\exp(-\lambda_m T)$ as $T \to \infty$. For the other asymptotics, first fix $k$ and let $h_{r,k} := h_r(\lambda_1, ..., \lambda_k)$ be our symmetric polynomials and $p_k = \prod_{i=1}^k \lambda_i$. Since $F_k(T)$ is the sum of a finite number of exponentials, its Taylor series at zero converges to $F_k(T)$ for all $T$. Moreover, (i) of Lemma \ref{lem:hypoexponential_derivative_properties} above tells us $F_k^{(k+r)}(0) = (-1)^r p_k h_{r,k}$. Thus, let
\[a_{r,k} = (-1)^r \frac{p_k h_{r,k}}{(k+r)!} T^{k+r}\]
\noindent denote the $(r+1)^{th}$ nonzero term in the Taylor series expansion about zero. Suppose that we approximate $F_k(T)$ by the first $(m+1)$ nonzero terms of its Taylor series about zero. We can measure the approximation error relative to the last term $a_{m,k}$ to see
\[\frac{F_k(T) - \sum_{r=0}^{m} a_{r,k}}{a_{m, k}} = \sum_{r=1}^\infty (-1)^r\frac{h_{m+r}/h_mT^r}{(k+m+r)_r}\]

\noindent where $(x)_r = x(x-1) \cdot ... \cdot (x-r+1)$ denotes the falling factorial. We will now bound the right hand side by a geometric series.

Note that trivially $h_{r+1} \leq h_r s_k$ as the former includes all monomials of degree $(r+1)$ in $\{\lambda_i\}_{i=1}^k$ with coefficient one, and the latter has all the monomials with coefficient at least one. Hence $h_{m+r} / h_m = \prod_{i=0}^{r-1} h_{m+i+1}/h_{m+i} \leq s_k^r$. Using this and the simple bound $(k + m + r)_r \geq k^r$ we get

\[\bigg|\sum_{r=1}^\infty (-1)^r\frac{h_{m+r}/h_mT^r}{(k+m+r)_r}\bigg| \leq \sum_{r=1}^\infty \left(\frac{s_kT}{k}\right)^r = \frac{s_kT/k}{1 - s_k T/k}\]

\noindent for $|s_kT/k| < 1$. When $T = o((s_k/k)^{-1})$, this tends to zero.  Thus in this regime
\[\frac{F_k(T) - \sum_{r=0}^{m} a_{r,k}}{a_{m, k}} \to 0,\]
 \noindent as desired. In particular
\[F_k(T) = (1 + o(1)) \frac{\prod_{i=1}^k \lambda_i}{k!} T^{k}.\]
\end{proof}

\subsubsection{Gap Asymptotics}
\label{app:gap-asymptotics}

\noindent The following statement gives us a sense of how fast our coalescent probabilities $g_{k,1}(T)$ will saturate as $k \to \infty$.  

\begin{theorem}[Gap Asymptotics] If $\sum_{i=1}^\infty \lambda_i^{-1} < \infty$, then $F_k \to F_\infty$ uniformly. Quantitatively, the gap $\Delta_k(T) = F_{k}(T) - F_\infty(T)$ has asymptotics
        \[\Delta_k = (f_\infty + o(1))r_k.\] 
\end{theorem}

  \begin{proof}
          Dini's theorem already guarantees that $F_k \to F_\infty$ uniformly. Here, we will get quantitative control on the rate of convergence. Note that (ii) of Lemma \ref{lem:hypoexponential_derivative_properties} implies the functions $f_k$ are uniformly bounded by $M_1$ and uniformly $M_2$-Lipschitz. Thus they form a uniformly equicontinuous family which is uniformly bounded.  Since $S_k \Rightarrow S_\infty$, as a consequence of \cite{Boos1985ConverseScheffe} converse to Scheffe's lemma, we see $f_{k} \to f_{\infty}$ uniformly on $[0, \infty)$. Now, note that by conditioning on the tail $R_k$, as $S_k \indep R_k$
        \[\P(S_\infty \leq T) = \int_0^\infty \P(S_k \leq T-r) \P(R_k \in dr) = \E F_k(T-R_k).\]
        \noindent Hence by Taylor expanding the CDF $F_{k}$ we get
        \[\Delta = \E_{ R_k} [F_{k}(T) - F_{k}(T-R_k)] = f_{k}(T) \cdot \E R_k + \frac{1}{2}\E \left(f'_{k}(\xi(R_k)) R_k^2\right),\]
        \noindent where $\xi(r) \in (T-r, T)$ for every $r$. Note that $\mathrm{Var}(R_k) = \sum_{i=k+1}^\infty\lambda_i^{-2}\leq r_k^2$, implying $\E R_k^2 \leq 2r_k^2$. Since $||f'_{k}||_\infty \leq M_2$ is uniformly bounded by (iii) and as $f_{k}(T) \to f_{\infty}(T)$ uniformly as $k \to \infty$
        \[\Delta_k(T) = f_\infty(T) r_k + (f_k(T) - f_\infty(T))r_k + O(r_k^2) = f_\infty(T) r_k + o(r_k).\]
        \noindent More specifically, since $M_2 = \max \{|F_1''(0)|, |F_2''(0)|\} =  \lambda_1\max(\lambda_1, \lambda_2)$
        \[||\Delta_k(T) - f_\infty r_k||_\infty \leq ||f_k - f_\infty||_\infty r_k + \lambda_1\max(\lambda_1, \lambda_2) r_k^2 = o(r_k)\]
        \noindent so the rate of convergence is uniform in $T$.
  \end{proof}

  Lemma \ref{lem:asymptotics_g_k1(T)} then just follows from the special case of Kingman's coalescent, where we have rates $\tilde{\lambda}_i = {i+1 \choose 2}$. Note the indexing is slightly different here than in the main paper: it is shifted down by one, but otherwise the same. Here we have
\[r_k = \sum_{i=k+1}^\infty \frac{1}{{i +1 \choose 2}} = \frac{2}{k+1}, \qquad s_k = \sum_{i=1}^k {i + 1 \choose 2} =  \frac{k(k+1)(k+2)}{6} = \Theta(k^3).\]
\noindent Moreover, a simple induction shows that $C_{1,k} = 3k/(k+2)$. Hence we get asymptotics

    \[F_k(T) \sim \begin{cases}
        1-\frac{3k}{k+2}\exp(-T), & T \uparrow \infty  \text{ and k fixed,}\\
        \frac{(k+1)!}{2^k} T^{k}, &  T = o(k^{-2}).\\
    \end{cases} \qquad \Delta_k(T) = (f_\infty(T) + o(1)) \cdot \frac{2}{k+1}\]

\subsubsection{Improvement Ratio Asymptotics}
\label{app:improvement-ratio-asymptotics}

\noindent In Section \ref{sec:estimated_growth_rates} above, we obtained the following asymptotic lower bound on the improvement factor of our balanced bound (Theorem \ref{thm:improved_bipartition_cover_bound_balanced}) over the original bound:
\[\beta_T := \frac{-\log(1-g_{u(T),1}(T))}{-\log(1-s(T))} \]
\noindent In this section, we develop small $T$ asymptotics for $\beta_T$. In particular:
\begin{lemma}[Small $T$ Improvements]
    As $T \downarrow 0$, we have
    \[\beta_T \sim \frac{\pi^2}{2T} \]
    \label{lem:improvement_ratio_asymptotics}
\end{lemma}
 \begin{proof}[Proof of Lemma \ref{lem:improvement_ratio_asymptotics}]
     
  Since $-\log(1-x) \sim x$ as $x \downarrow 0$
\[\beta_T \sim 1 + \frac{\Delta(T)}{s(T)}, \qquad \Delta(T) := g_{u(T), 1}(T) - s(T). \]
\noindent Lemma \ref{lem:asymptotics_g_k1(T)} and Corollary \ref{corr:upperbound_lineages_balanced} imply as $T \downarrow 0$
\[\frac{\Delta(T)}{s(T)} \sim \frac{2f_{S_\infty}(T)}{s(T)} \cdot \frac{1}{u(T)} \sim T \cdot r(T),\]
\noindent where $r := d/dt \log(s(t))$ is the reversed hazard rate of $S_\infty$. To control this, we develop asymptotics for  $K(x):=\log \E \exp(-xS_\infty)$ and use a classical Tauberian theorem to relate this to the log CDF. By independence, for $x > 0$
\[K(x) = \log \left(\prod_{n=2}^\infty \frac{\lambda_n}{\lambda_n + x}\right) = -\log\left(\prod_{n=2}^\infty \left[1 + \frac{x}{\lambda_n}\right]\right).\]
\noindent Since $(n-1)^2 \leq 2\lambda_n \leq n^2$, using the infinite product expansion of $\sinh$, we see
\[0 \leq K(x) + \log\left(\frac{\sinh(\pi \sqrt{2x})}{\pi \sqrt{2x}}\right) \leq \log(1+2x).\]
\noindent Hence as $\log(y^{-1} \sinh(y)) \sim y$ as $y \to \infty$, we see $K(x) \sim -\pi\sqrt{2x}$ as $x \to \infty$. Applying Corollary 2 of \cite{cadena2015notetauberiantheoremsexponential} with $A = 1, P(x) = s(x), B = -\pi^2/2,$ and $\beta = -1$ we see this implies that $\log(s(T)) \sim -\pi^2/2 T^{-1}$ as $T \downarrow 0$. 

 Let $h(t) := \log s(t)$. Pr\'ekopa's theorem implies that $S_k$ is log-concave for all $k$ and, as log-concavity is preserved under weak limits, so is $S_\infty$. Hence its distribution function $s(T)$ is log-concave, and thus $h(t)$ is concave. Thus for fixed $\epsilon \in (0,1)$, its derivative is bounded by the two secant slopes
\[\frac{h((1+\epsilon)t) - h(t)}{\epsilon t} \leq h'(t) \leq \frac{h(t) - h((1-\epsilon)t)}{\epsilon t}.\]
\noindent Our asymptotics for $h$ then imply
\[\frac{1 + o_\epsilon(1)}{1+\epsilon} \leq \frac{h'(t)}{\pi^2/2t^2} \leq \frac{1+o_\epsilon(1)}{1-\epsilon},\]
\noindent and thus
\[\frac{1}{1+\epsilon} \leq\lim\inf_{t \downarrow 0} \frac{h'(t)}{\pi^2/2t^2} \leq \lim\sup_{t \downarrow 0} \frac{h'(t)}{\pi^2/2t^2} \leq \frac{1}{1-\epsilon} \]
\noindent Sending $\epsilon \downarrow 0$ shows $r(t) \sim \pi^2/(2t^2)$ as desired.
\end{proof}

\noindent Note that above we showed
\[s(T) = \exp\left(-\frac{\pi^2}{2} T^{-1} (1 + o(1)) \right).\]
Since for typical trees (e.g. Yule trees) we have $T_{\mathrm{min}} = O(k^{-1})$, this heuristically tells us the original bound is $O(\log(k) \exp(k))$.

\subsubsection{Bound Asymptotics (Lemma \ref{lem:bound_asymptotics})}
\label{app:bound-asymptotics}

Now that we have analyzed the core function $g_{k,1}(T)$ we are ready to analyze our bounds. Let $\kappa_{q,k} = -\log\left((1-q)/(k-3)\right)$ be the coefficient that occurs in the numerator of all our bounds. Moreover, let $s(T) := \P(S_\infty \leq T)$ be our limiting CDF, $\Delta_k(T) := g_{k,1}(T) - s(T)$ our gap, and $c(T) := 2f_{S_\infty}(T)$ our gap's asymptotic constant. Recall that $M_o(k, T)$ denotes the original bound from Theorem \ref{thm:uricchios_bipartition_cover_bound} and $M_b(k,T)$ our balanced bound Theorem \ref{thm:improved_bipartition_cover_bound_balanced}.

 We start with $M_o$. Note Lemma \ref{lem:asymptotics_g_k1(T)} implies
\begin{equation}
    \log(1-g_{k,1}(T)) \sim \begin{cases} - T, & T \uparrow \infty, \\ 
 -\frac{k!}{2^{k-1}} T^{k-1}, &  T = o(k^{-2}).
\end{cases}
\label{eqn:asymptotics_T_mapsto_log(one_minus_g_k1(T))}
\end{equation}
\noindent From this

\[M_{o}(k,T) \sim \begin{cases}
    \frac{\kappa_{q,k}}{T}, & T \uparrow \infty, \\
    \frac{2^{k-3}\kappa_{q,k}}{(k-2)! T^{k-3}}, & T = o(k^{-2}).
\end{cases}\]

\noindent In particular, we recover the small $T$ bound found in Appendix $B$ of \cite{shekhar2018_how_many_genes_are_enough}, although here we exactly quantify the asymptotic constant.

  For the fixed $T$ and asymptotic $k$ bound, we Taylor expand $\log(x)$ about $1 - s$ to observe:\\
\[\log(1 - g_{k,1}(T) )= \log(1-s)-\sum_{n=1}^\infty \frac{\Delta^n}{n(1-s)^n} = \log(1-s) - \frac{\Delta}{1-s} - O(\Delta^2).\]
\noindent Hence this and Lemma \ref{lem:asymptotics_g_k1(T)} imply as $k \to \infty$
\[M_o(k,T) = \frac{\kappa_{q,k}}{\left(\frac{c(T)}{1-s} + o(1)\right) \cdot \frac{1}{k} - \log(1-s)} \sim \frac{\kappa_{q,k}}{-\log(1-s)}.\]

\noindent For $M_b(k,T)$, note as $g_{k,1}(T)$ is decreasing in $k$ we have
\[\sum_{\ell=2}^{k-2} (1-w_\ell(T))^n \geq (k-3)(1-g_{2,1}(T))^n = (k-3)\exp(-nT).\]
\noindent Hence $M_b(k,T)\geq \kappa_{q,k}T^{-1}$.
As $k \mapsto g_{k,1}(T)$ is decreasing and convex by Lemma \ref{lem:g_i1_is_convex}
    \[\sum_{\ell=2}^{k-2} (1-\E g_{U_\ell, 1}(T))^n \leq (k-3) (1- g_{\E U_{k-2}, 1}(T))^n.\]
\noindent Applying Corollary \ref{corr:upperbound_lineages_balanced}
\[M_b(k,T) \leq \frac{\kappa_{q,k}}{-\log(1-g_{u(T),1}(T))}, \quad u(T) := \frac{2-\exp(-T/2)}{1-\exp(-T/2)}.\]
\noindent Combining these two previous results
\[\frac{\kappa_{q,k}}{T} \leq M_{b}(k,T) \leq \frac{\kappa_{q,k}}{-\log(1-g_{u(T), 1}(T))}.\]
\noindent Since $M_b(k,T) \leq M_o(k,T)$, the lower bound shows us $M_b(k,T) \sim \kappa_{q,k} T^{-1}$ as $T \to \infty$. Since $u(T) \sim 2/T$ as $T \downarrow 0$, the upper bound does not provide any provable improvement in the regime $k^2T = o(1)$ that we analyzed previously. More careful analysis of the regime $kT = o(1)$ is likely necessary to observe a noticeable improvement, but this is outside the scope of our current analysis.

For fixed $T$, the upper bound implies that the asymptotic improvement factor of $M_b(k,T)$ over $M_o(k,T)$ is at least
\[\beta_T := \frac{\log(1-g_{u(T),1}(T))}{\log(1-s(T))}.\]
\noindent In Lemma~\ref{lem:improvement_ratio_asymptotics}, we show that $\beta_T \sim \pi^2/(2T)$ as $T \downarrow 0$. Thus, in this regime, the improvement factor is bounded below by a quantity asymptotic to $\pi^2/(2T)$.

Repeating our above analysis shows that \textit{any} bound built using the union bound and the coalescent function $g_{k,1}(T)$ cannot be asymptotically better than $\log(k)/T$. This is true, \textit{even if the descendant counts $\alpha_i$ are known}, as we can still upper bound $g_{\alpha_i, 1}(T) \leq g_{2,1}(T)$. Hence the asymptotics in $k$ are fundamentally limited to $\Theta(\log(k))$ by our union bound argument, and we can only hope to improve upon the constant down to the limit $T^{-1}$. Still, we see some finite-sample improvements empirically, as we observed in Figure \ref{fig:improvement_ratios}.

\section*{Statements and Declarations}

\noindent \textbf{Acknowledgments:} The author would like to thank Professor Steven Evans and Daniel Raban for their guidance and helpful comments in improving the exposition of this article.\\

\noindent\textbf{Competing interests:} No competing interests to declare.\\

\noindent \textbf{Funding}: This research did not receive any specific grant from funding agencies in the public, commercial, or not-for-profit sectors.\\

\noindent\textbf{Data availability:} Code used for the simulations and numerical experiments is available at \url{https://github.com/zackmcnulty/msc_bipartition_cover}. No external datasets were analyzed in this study. All results involving data were generated by simulation.\\

\noindent \textbf{Generative AI use and AI-assisted technologies}: During the preparation of this work the author used OpenAI Codex to assist with debugging and 
reviewing simulation code. Additionally, this tool was used in the above document to aid in typesetting, copy-editting, and improving the exposition. 
After using this tool/service, the author reviewed and edited the content as needed and takes full responsibility for the content of the published article.

\newpage 
\bibliographystyle{plainnat}
\bibliography{references}
\end{document}